\documentclass[11pt]{article}

\usepackage{amssymb,amsfonts,amsthm,amsmath, upref}

\usepackage{booktabs}

\usepackage{graphicx, verbatim, centernot}
\usepackage{xcolor}

\usepackage{multicol}
\usepackage{enumitem}

\usepackage[a4paper, hmargin=2cm, vmargin={3cm, 3cm}]{geometry}
\usepackage{fancyhdr}
\usepackage[linktocpage=true]{hyperref}
\hypersetup{
	colorlinks=true,
	linkcolor=blue,
	citecolor=purple,
	filecolor=magenta,      
	urlcolor=purple,
}

\newtheorem{theorem}{Theorem}

\newtheorem{lemma}[theorem]{Lemma}

\newtheorem{proposition}[theorem]{Proposition}

\newtheorem{claim}[theorem]{Claim}

\theoremstyle{definition}
\newtheorem*{definition*}{Definition}
\newtheorem{definition}[theorem]{Definition}

\newcommand{\theoremname}{testing}
\theoremstyle{remark}
\newtheorem*{remark*}{Remark}
\newtheorem{remark}[theorem]{Remark}

\numberwithin{theorem}{section}

\def\eps{{\varepsilon}}
\def\La{{\Lambda}}
\def\Ga{{\Gamma}}

\newcommand{\bR}{\mathbb R}

\newcommand{\bZ}{\mathbb Z}

\newcommand{\bE}{\mathbb E}
\newcommand{\cE}{\mathcal{E}}

\newcommand{\cB}{\mathcal B}

\newcommand{\cG}{\mathcal G}

\newcommand{\Z}{\mathbb{Z}}
\newcommand{\Cov}{\operatorname{Cov}}

\newcommand{\one}{\boldsymbol{1}}

\newcommand{\footremember}[2]{%
    \footnote{#2}
    \newcounter{#1}
    \setcounter{#1}{\value{footnote}}%
}
\author{%
  Sourav Chatterjee\footremember{A}{Department of Mathematics and Department of Statistics, Stanford University, USA. Email: {\tt souravc@stanford.edu} }%
  \and Oren Yakir\footremember{B}{Department of Mathematics, Massachusetts Institute of Technology, USA. Email: {\tt oren.yakir@gmail.com}}}
  
\title{Correlation decay for U(1) lattice Higgs theory: the case of small mass}
\date{}

\begin{document}

	\maketitle

	\begin{abstract}
    We study the lattice Yang--Mills--Higgs model with inverse gauge coupling $\beta>0$ and Higgs length $\alpha>0$, in the ``complete breakdown of symmetry" regime. For lattice dimension $d\ge 2$ and (abelian) gauge group U(1), we prove that for any $m>0$, if $\alpha= m\beta$ and $\beta$ is large enough, the model exhibits exponential decay of correlations. This extends the classical result of Osterwalder and Seiler~\cite{Osterwalder-Seiler},
    who required in addition that $m$ be sufficiently large. Our result also verifies a phase diagram from the physics literature, predicted by Fradkin and Shenker~\cite{Fradkin-Shenker}. 
    The proof is based on the Glimm--Jaffe--Spencer cluster expansion around a massive Gaussian field, following the approach of Balaban et al.~\cite{Balaban-Imbrie-Jaffe-Brydges}.
	\end{abstract}


    \pagestyle{plain}
    
    \section{Introduction}
    \label{sec:introduction}
    In this paper we study the validity of the Higgs mechanism for a certain abelian Yang--Mills lattice model. Let $d\ge 2$ and let $\Lambda \subset \bZ^d$ denote a finite box in the lattice $\bZ^d$. We denote by $E(\Lambda)$ and $P(\Lambda)$ the set of (positively oriented) edges and plaquettes in $\Lambda$, respectively. The Yang--Mills--Higgs lattice measure with gauge group U(1), inverse coupling constant $\beta>0$ and Higgs length $\alpha>0$ is the Gibbs measure on $\boldsymbol{\theta} =(\theta_e)_{e\in E(\Lambda)} \subset [-\pi,\pi]^{E(\Lambda)}$ which corresponds to the action
    \begin{equation}
        \label{eq:intro-hamiltonian_with_alpha_beta}
        \boldsymbol{\theta} \mapsto \beta \sum_{p\in P(\Lambda)} \cos\big({\sf d} \theta_p\big) + \alpha \sum_{e\in E(\Lambda)} \cos\big(\theta_e\big) \, .
    \end{equation}
    Here ${\sf d}$ is the lattice exterior derivative. That is, for a plaquette $p\in P(\La)$ which consists of the oriented edges $e_1,e_2,e_3,e_4$ in counter clockwise order (see Figure~\ref{figure:plaquette}), we have 
    \[
    {\sf d} \theta_p = \theta_{e_1} + \theta_{e_2} - \theta_{e_3} - \theta_{e_4} \, .
    \]
    \noindent
    The action~\eqref{eq:intro-hamiltonian_with_alpha_beta} with $\alpha = 0$ is nothing but the pure lattice Yang--Mills action with gauge group U(1). With $\alpha>0$, it corresponds to the Yang--Mills gauge field coupled to a Higgs field in the so-called ``complete breakdown of symmetry" regime. Namely, when the Higgs field takes values in the gauge group (which is U(1), in our case), and the gauge field acts on it via the trivial representation. The action~\eqref{eq:intro-hamiltonian_with_alpha_beta} is then obtained by unitary gauge fixing, which removes the Higgs field from consideration. While we will not motivate this choice of model further, we remark that it is quite a popular choice both in the mathematics and in the physics literature for studying the Higgs mechanism in (abelian or non-abelian) Yang--Mills models; see for example~\cite{Banks-Rabinovici, Chatterjee-HiggsMechanism, Fradkin-Shenker, Osterwalder-Seiler, Seiler-book, Shen-Zhu-Zhu}.
    	\begin{figure}
		 	\begin{center}	\scalebox{0.4}{\includegraphics{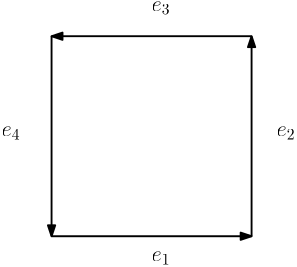}}
		 	\end{center}
            \caption{A plaquette consisting of four oriented edges $e_1,e_2,e_3,e_4$.}
                    \label{figure:plaquette}
		\end{figure}
    \subsection{Main result}
    In what follows, we will re-parametrize the Yang--Mills--Higgs action~\eqref{eq:intro-hamiltonian_with_alpha_beta}. Indeed, in this paper we will study the case $\alpha = m \beta$, where $m>0$ is an arbitrary constant that will be kept fixed throughout.\footnote{In fact, our result also applies when $m$ is allowed to decrease slowly as $\beta\to\infty$, see Remark~\ref{remark:mass_tending_to_Zero} below.}
    We set
    \begin{equation}
        \label{eq:intro_hamiltoniam_with_beta_and_mass}
        \mathcal{H}_{\Lambda}(\boldsymbol{\theta}) = \sum_{p\in P(\Lambda)} \big[1 - \cos\big({\sf d} \theta_p\big)\big] + m \sum_{e\in E(\Lambda)} \big[1- \cos\big(\theta_e\big)\big] \, .  
    \end{equation}
    For finite $\Lambda\subset \Z^d$, the \emph{Yang--Mills--Higgs measure} on $\Lambda$ with free boundary condition is given by
    \begin{equation}
        \label{eq:intro-def_of_YMH_measure}
        {\rm d}\mu_{\Lambda}^{\sf YMH}(\boldsymbol{\theta}) =  \frac{1}{Z_{\Lambda}^{\sf YMH}} \exp\Big(-\beta \, \mathcal{H}_{\Lambda}(\boldsymbol{\theta})\Big) \prod_{e\in E(\Lambda)} \one_{\{|\theta_e| \le \pi\}} \, {\rm d}\theta_e \, ,
    \end{equation}
    where $Z_{\Lambda}^{\sf YMH}$ is the partition function that normalizes $\mu_{\Lambda}^{\sf YMH}$ into a probability measure. We denote by $\bE_{\Lambda}^{\sf YMH}[\cdot]$ the expectation operator with respect to the probability measure~\eqref{eq:intro-def_of_YMH_measure}. Our main result is a mass gap for this particular lattice Yang--Mills--Higgs model, which holds for arbitrary small $m>0$ as long as $\beta$ is large enough (depending on $m$). We denote by $\text{\normalfont dist}$ is the graph distance in $\bZ^d$.
    \begin{theorem}
        \label{thm:exponential_clustering_in_lattice_YMH}
        For any $m>0$ and for all $\beta\ge \beta_0(d,m)$ large enough, there exist constants $C,c>0$ which are uniform in the size of $\Lambda_n = \{-n,\ldots,n\}^d \subset \bZ^d$, so that for all $x,y\in E(\Lambda_n)$ we have
        \[
        \Big|\bE_{\Lambda_n}^{\sf YMH}\big[\theta_x \theta_y\big] - \bE_{\Lambda_n}^{\sf YMH}\big[\theta_x \big]\cdot \bE_{\Lambda_n}^{\sf YMH}\big[ \theta_y\big]\Big| \le C e^{-c  \, \text{\normalfont dist}(x,y)} \, .
        \]
    \end{theorem}
   	\begin{figure}
		 	\begin{center}	\scalebox{0.25}{\includegraphics{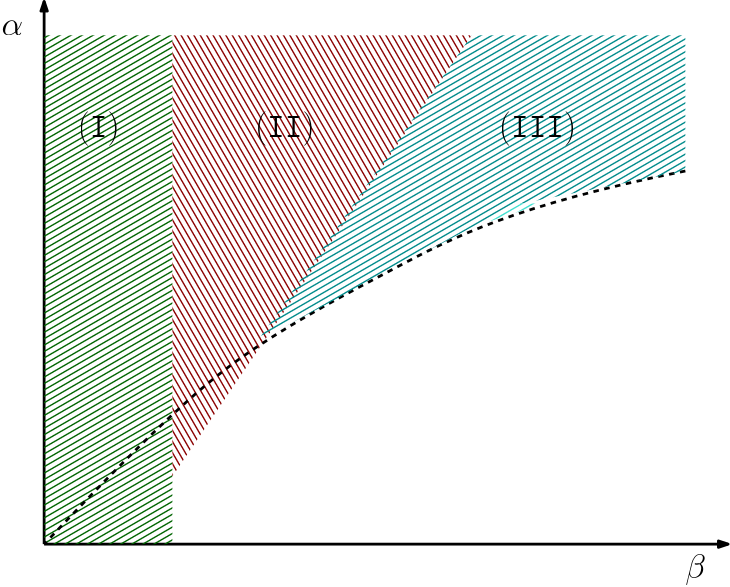}}
		 	\end{center}
            \caption{The Fradkin--Shenker phase diagram. The colored regions indicate phases where the model exhibits a mass gap. In domains {\tt (I)} and {\tt (II)}, the presence of a mass gap was proved in earlier works (see discussion in the text). The existence of a mass gap in the region {\tt (III)} is the new contribution of this paper. 
            The dashed line represents a graph of a certain concave function. }
            \label{figure:phase_diagram}
		\end{figure}

    \subsection{Related works}

    In particle physics, the Higgs mechanism is a way to generate a `mass gap' (i.e., exponential decay of correlations in the theory) via spontaneous symmetry breaking. In our particular model, this is effectively obtained via the term 
    \[
    m \sum_{e\in E(\Lambda)} \big[1- \cos\big(\theta_e\big)\big]
    \]
    in the action~\eqref{eq:intro_hamiltoniam_with_beta_and_mass}. Thus, it is natural to ask for which values of $m>0$ the mechanism ``works" and a mass gap is present.  
    Indeed, the main motivation for our work comes from the influential physics paper by Fradkin and Shenker~\cite{Fradkin-Shenker}, which investigates the phase diagram of the model~\eqref{eq:intro-hamiltonian_with_alpha_beta} in the $(\alpha,\beta)$ plane. In~\cite{Fradkin-Shenker}, the authors predict the presence of a mass gap whenever $\alpha$ is bigger than some concave function of $\beta^{-1}$, provided that both parameters are large; see~\cite[Figure~2]{Fradkin-Shenker}. In particular, this suggests that a mass gap should persist even when $m= \alpha/\beta$ is arbitrarily small. See also the work of Banks and Rabinovici~\cite{Banks-Rabinovici} for a further discussion on this phase diagram and its physical implications. 

    From the mathematical standpoint, the existence of a mass gap for the Yang--Mills--Higgs measure~\eqref{eq:intro-def_of_YMH_measure} has been rigorously established only in certain regimes: 
    \begin{enumerate}
        \item[{\tt (I) }] For small $\beta$ and any $m\ge 0$ (see~\cite{Osterwalder-Seiler, Seiler-book}). 
         \item[{\tt (II)}] For both $\beta$ and $m$ sufficiently large (see~\cite{Seiler-book} and~\cite[Appendix]{Fradkin-Shenker}). 
    \end{enumerate}
    Theorem~\ref{thm:exponential_clustering_in_lattice_YMH} complements these results by proving the existence of a mass gap for all $m>0$, assuming that $\beta$ is sufficiently large depending on $m$, thereby providing a rigorous verification of part of the Fradkin--Shenker prediction. Moreover, our analysis shows that $m$ can in fact tend to zero as $\beta\to\infty$ while still maintaining the mass gap; see Remark~\ref{remark:mass_tending_to_Zero} for a more detailed discussion on this point. 
    
    As we already mentioned, when $m=0$ the measure~\eqref{eq:intro-def_of_YMH_measure} reduces to the pure U(1) Yang--Mills lattice model. For this model, the celebrated Berezinskii–Kosterlitz–Thouless (BKT) transition is known to occur for lattice dimension $d=4$. In the regime of small $\beta$, correlations decay exponentially as was proved rigorously by Osterwalder and Seiler~\cite{Osterwalder-Seiler} (and holds in any dimension $d\ge 2$). We also mention the recent work~\cite{Shen-Zhu-Zhu} of Shen, Zhu and Zhu, which obtained similar findings for a wide class of lattice Yang-Mills-Higgs type models, using PDE techniques. 
    On the other hand, for $\beta$ sufficiently large the correlations decay only polynomially fast in $d=4$, as proved by Fr\"ohlich and Spencer~\cite{Frohlich-Spencer} (see also Guth~\cite{guth80}). We also mention the work by Garban and Sep\'ulveda~\cite{Garban-Sepulveda-IMRN2023} for finer asymptotic properties of the U(1) Yang--Mills model on $\bZ^4$, in the large $\beta$ regime. 

    The Higgs mechanism for abelian gauge fields on a lattice has previously been established for a related (but different) Yang--Mills--Higgs type model, in the work of Balaban, Imbrie, Jaffe, and Brydges~\cite{Balaban-Imbrie-Jaffe-Brydges}. There, the authors study a U(1) gauge field coupled to an $\mathbb{R}^2$-valued Higgs field, whose marginal is sampled from the so-called ``Mexican hat potential''. In~\cite{Balaban-Imbrie-Jaffe-Brydges}, the authors employ a gauge fixing that preserves the Higgs field, in contrast to the model~\eqref{eq:intro-hamiltonian_with_alpha_beta}, where the Higgs field has already been eliminated via gauge fixing, corresponding to the so-called ``complete breakdown of symmetry'' regime. Despite this difference,~\cite{Balaban-Imbrie-Jaffe-Brydges} also establishes a mass gap for their model in a different parameter regime. In fact, our proof techniques are closely related, and our analysis in several parts of this paper is similar to that of~\cite{Balaban-Imbrie-Jaffe-Brydges}. In some sense, the elimination of the Higgs field makes our argument slightly less technical and perhaps more accessible, since~\cite{Balaban-Imbrie-Jaffe-Brydges} must track the joint behavior of the gauge and Higgs fields throughout their work. Both~\cite{Balaban-Imbrie-Jaffe-Brydges} and the present work ultimately rely on a convergent cluster expansion due to Glimm, Jaffe, and Spencer~\cite{Glimm-Jaffe-Spencer}; see Section~\ref{subsection:idea_of_the_proof} for a high-level overview of the approach. We also refer the reader to the survey~\cite{Borgs-Nill} for additional background on phase diagrams of the model studied in~\cite{Balaban-Imbrie-Jaffe-Brydges}.

    There have also been recent developments in the study of the Higgs mechanism in abelian gauge fields. Most closely related to our work is the recent paper~\cite{Chatterjee-HiggsMechanism} by the first named author. In~\cite{Chatterjee-HiggsMechanism}, among other things, a continuum Gaussian scaling limit for the field sampled from~\eqref{eq:intro-def_of_YMH_measure} is obtained, in the joint limit $\beta\to\infty$, $m\to 0$, and lattice spacing $\to 0$ simultaneously at a suitable rate. It is further shown that the limiting field is massive, suggesting that mass can be generated even when $m$ is (effectively) small. In contrast, Theorem~\ref{thm:exponential_clustering_in_lattice_YMH} shows that mass is in fact generated already in the lattice model for all $m>0$, provided $\beta$ is sufficiently large. 

    In another direction, there has also been recent work on the decay of correlations in the Yang-Mills-Higgs lattice models with finite (abelian or non-abelian) gauge group. We refer the interested reader to~\cite{Adhikari-Cao, Frosstrom, Frosstrom-Lenells-Viklund} and the references therein.


    \subsection{Rough idea of the proof}
    \label{subsection:idea_of_the_proof}

    Theorem~\ref{thm:exponential_clustering_in_lattice_YMH} follows from a convergent cluster expansion, which is uniform in the volume of the lattice square $\La_n \subset \bZ^d$. A wrinkle to the story is that unlike in more ``standard" cluster expansions, typically expanded around a ground state or an i.i.d.\ field, here we show that the model~\eqref{eq:intro-def_of_YMH_measure} is, in fact, a small perturbation of a certain massive Gaussian field with non-trivial correlations. We remark that the work of Osterwalder and Seiler~\cite{Osterwalder-Seiler} (see also~\cite[Appendix]{Fradkin-Shenker}) applied a cluster expansion for the same model~\eqref{eq:intro-def_of_YMH_measure} but around an i.i.d.\ field, which is the chief reason why their domain of convergence is smaller, ultimately requiring that $m$ must  be large as well. 
    
    Indeed, for $\beta$ large, we expect typical configurations sampled from~\eqref{eq:intro-def_of_YMH_measure} to be small. On a formal level, expanding the cosine 
    \[
    \cos(\theta) \approx 1- \frac{\theta^2}{2} \, ,
    \]
    yields the approximation   
    \begin{equation}
        \label{eq:intro-gaussian_approximation_of_YMH_action}
        \mathcal{H}_{\Lambda}(\boldsymbol{\theta}) \approx \frac{1}{2} \sum_{p\in P(\Lambda)} \big({\sf d} \theta_p\big)^2 + \frac{m}{2} \sum_{e\in E(\Lambda)} \big(\theta_e\big)^2 \, . 
    \end{equation}
    The right-hand side of~\eqref{eq:intro-gaussian_approximation_of_YMH_action} is an action of a certain Gaussian field, namely, the Proca field.
    In fact, its continuum analogue is exactly the Gaussian scaling limit which appeared in~\cite{Chatterjee-HiggsMechanism} and was mentioned in the previous section on related works. 
    Heuristically, the approximation~\eqref{eq:intro-gaussian_approximation_of_YMH_action} suggests that the Yang--Mills--Higgs measure~\eqref{eq:intro-def_of_YMH_measure} is a small perturbation of the Proca field as $\beta \to \infty$. To carry out this reasoning formally, we employ a technique attributed to Glimm, Jaffe, and Spencer~\cite{Glimm-Jaffe-book, Glimm-Jaffe-Spencer}, which enables one to cluster expand around a given massive Gaussian field. As a first step, we show (Lemma~\ref{lemma:proca_field_is_massive} below) that the Proca field is massive, with correlations between $\theta_x$ and $\theta_y$ decaying like
    \[
    \exp\big( - c_d \, m \, \text{dist}(x,y)\big)
    \]
    uniformly for $\beta \ge 1$, where $c_d > 0$ is a constant depending only on the lattice dimension $d \ge 2$. With this in hand, we may take    
    \begin{equation}
    L \ge L_0(m,d),
    \end{equation}
    to be sufficiently large, and partition the lattice domain $\Lambda$ into $L$-blocks. That is, lattice boxes of side length $L$, indexed by points of $L \bZ^d$. Naively, one may think that taking $L= C_d \, m^{-1}$ for some sufficiently large $C_d$ is sufficient, but due to various technical aspects of the proof we in fact need to take $L$ to be much larger. We then consider a new Gaussian field which has the same law as the Proca field inside each individual $L$-block, but is otherwise completely uncorrelated. The Glimm--Jaffe--Spencer cluster expansion allows us to expand the Proca field around this block-independent Gaussian field, as $L$ grows large.

    Next, we show that the error terms resulting from the approximation~\eqref{eq:intro-gaussian_approximation_of_YMH_action} can also be approximately factored into independent contributions from the different $L$-blocks, provided that $L$ is large and $\beta \to \infty$. A technical caveat is that the Gaussian approximation~\eqref{eq:intro-gaussian_approximation_of_YMH_action} only holds when the Yang--Mills--Higgs field variables remain small. While this is typically the case, there may be regions in $\Lambda$ where the field takes atypically large values, making the approximation invalid. We treat such regions separately, first showing that they do not percolate, and only then expanding around the Gaussian field in the remaining part of the domain.

    The cluster expansion is one of the oldest and most widely used techniques in statistical mechanics, dating back at least to the work of Mayer~\cite{Mayer} in 1937. A standard introduction to the basics can be found in the book by Friedli and Velenik~\cite[Chapter~5]{Friedli-Velenik_book}, while a more advanced treatment, including a general framework for cluster expansions around Gaussian fields, can be found in Brydges's excellent lecture notes~\cite{Brydges-course}. We also note that perturbative arguments around (massive or massless) Gaussian fields are a recurring theme in the mathematical literature on gauge theory. Notable examples include~\cite{Balaban-Imbrie-Jaffe-Brydges, Chatterjee-HiggsMechanism, Frohlich-Spencer,Garban-Sepulveda-IMRN2023, Glimm-Jaffe-Spencer, Osterwalder-Seiler, Seiler-book}, to name a few.

    \begin{remark}
    \label{remark:mass_tending_to_Zero}
    Our proof of Theorem~\ref{thm:exponential_clustering_in_lattice_YMH} is robust enough to yield the following (slightly stronger) statement: there exists $\beta_0(d)$ so that for all $\beta\ge \beta_0(d)$ and for all $m \ge (\log\beta)^{-1}$, the correlations in~\eqref{eq:intro-def_of_YMH_measure} decay exponentially with the distance. In other words, the proof remains valid even if we take the mass parameter $m$ to decay slowly as a function of $\beta$, for instance as $m = (\log \beta)^{-1}$. We will not provide full details of this extension, but we note that it requires choosing the coarse scale $L$ to grow like $C_d \log \beta$ for some sufficiently large constant $C_d$. Aside from that, the structure of the proof remains essentially unchanged. 

    We find this observation interesting enough to discuss particularly because it implies that the domain in the $(\alpha,\beta)$ plane where a mass gap is known to hold lies above a concave curve. More precisely, it shows that the threshold function
    \[
    \psi(\beta) = \inf \big\{ \alpha \ge 0 \, :  \, \text{the model~\eqref{eq:intro-hamiltonian_with_alpha_beta} with $(\alpha,\beta)$ exhibits a mass gap} \big\} 
    \]
    lies below the graph of the concave function $x\mapsto x/\log x$. See Figure~\ref{figure:phase_diagram} for a visual demonstration, where the dashed line represents this bound. We do not know what is the actual asymptotic behavior of $\psi(\beta)$ as $\beta\to\infty$, and we are not aware of any predictions on this question in the physics literature. It is also worth noting that the actual statement of Theorem~\ref{thm:exponential_clustering_in_lattice_YMH} in fact shows that $$ \lim_{\beta\to \infty} \frac{\psi(\beta)}{\beta} =0 \, . $$ 
    \end{remark}
    
    \subsection{Notation and lattice terminology}
    We end the introduction by fixing some terminology that will be used freely throughout the paper. Throughout, we let 
    \begin{equation}
            \label{eq:choice_of_L}
            L \ge L_0(m,d) \, ,
        \end{equation}
    to be a large enough even number, depending only on the lattice dimension $d\ge 2$ and the mass term $m$. Our cluster expansion will be preformed in the coarse lattice $L \bZ^d$. Indeed, every vertex $v \in L\bZ^d$ will be identified with its corresponding $L$-block, given by
    \[
    Q(v) = v + \{-L/2,\ldots,L/2-1\}^d \, .
    \]
    Note that these $L$-block are in fact disjoint, and their union gives the whole $\bZ^d$. We denote by $\La^\prime \subset L\bZ^d$ the minimal set such that 
    \begin{equation} \label{eq:def_of_la_prime}
        \La_n \subset \bigcup_{v\in \La^\prime} Q(v) \, , 
    \end{equation}
    where we recall that $\La_n= \{-n,\ldots,n\}^d$ is the lattice domain from Theorem~\ref{thm:exponential_clustering_in_lattice_YMH}. While for most vertices $v\in \La^\prime$ we have $Q(v)\subset \La_n$, there could be some (boundary) vertices for which it is not the case. With a slight abuse of notation, we will redefine the $L$-blocks for those vertices as
    \[
    \widetilde Q(v) = Q(v) \cap \La_n  \, ,
    \]
    but continue denoting all $L$-blocks by $\{Q(v)\}_{v\in \La^\prime}$ for notational simplicity. Given any two neighboring vertices $v,u\in L\bZ^d$, we define the \emph{face} between them as the set of edges
    \[
    {\sf F}(u, v) = \big\{ (x,y) \in E(\mathbb{Z}^d)  :  x\in Q(u) \, ,\   y\in Q(v) \big\}.
    \]
    For a non-empty subset $\mathcal{S}\subset L\bZ^d$, we denote by $\cE(\mathcal{S})$ as the set of all edges contained the corresponding union of $L$-blocks, namely  
    \[
    \cE(\mathcal{S}) = E\Big(\bigcup_{v\in \mathcal{S}} Q(v) \Big) \, .
    \]
    We will also denote by $P^\prime(\mathcal{S})$ as the set of all faces contained in $\cE(\mathcal{S})$. Given any subset of faces $\Ga\subset P^\prime(L\bZ^d)$, we set 
    \begin{equation}\label{xsdef}
        X(\Ga) = \big\{ v\in L\bZ^d \, : \, \exists u\in L\bZ^d \, : {\sf F}(u,v) \in \Ga  \big\} \, .
    \end{equation}
    We say that two non-empty subsets $\mathcal{S}_1,\mathcal{S}_2\subset L\bZ^d$ \emph{touch} if there exists $u\in \mathcal{S}_1$ and $v\in \mathcal{S}_2$ which are neighbors on $L\bZ^d$. We say that $\mathcal{S}\subset L\bZ^d$ and $\Ga\subset P^\prime(L\bZ^d)$ touch if $\mathcal{S}, X(\Ga)\subset L\bZ^d$ touch, in the sense defined above. In general, in this paper we will mostly identify subsets of the lattice with its edge elements. In particular, for a finite subset $A\subset \bZ^d$ we will consider the \emph{edge boundary} in this paper, given as
    \[
    \partial A = \big\{ (u,v)\in E(\bZ^d) \, : \, u\in V(A) \, , \ v\in V(\bZ^d\setminus A) \big\}.
    \]
    Finally, we introduce some asymptotic notations that will be convenient for us. We write $A\lesssim B$ or $A = O(B)$ if there exists a constant $C>0$, that may depend only on the dimension $d\ge 2$ and the constant $m>0$, such that $A \le C B$. Throughout, $C,c>0$ will be large and small constants, respectively, that may depend only on $d$ and $m$, and may change from line to line. We write $C_L,c_L>0$ to denote such constants that may depend on $L$ as well, and again may change from line to line. 
   
    \section{Setting up the polymer expansion}
    \label{sec:setting_up_cluster_expansion}
    To prove Theorem~\ref{thm:exponential_clustering_in_lattice_YMH}, we will consider a modified partition function for the model~\eqref{eq:intro-def_of_YMH_measure}, by inserting ``source" terms in the edge variables $x,y\in E(\Lambda)$. We set 
    \begin{equation}
        \label{eq:def_of_V_s}
        V_{\sf s}(t_x,t_y;\boldsymbol{\theta}) = V_{\sf s}(\boldsymbol{\theta}) = t_x \theta_x + t_{y}\theta_y
    \end{equation}
    where $t_x,t_y$ are complex parameters with $\max\{|t_x|,|t_y|\} \le \beta^{-10}$ (say). Defining 
    \[
    \mathcal{H}_{\Lambda,{\sf s}}(\boldsymbol{\theta}) = \mathcal{H}_{\Lambda}(\boldsymbol{\theta}) + V_{\sf s}(\boldsymbol{\theta})
    \]
    with $\mathcal{H}_{\Lambda}$ given by~\eqref{eq:intro_hamiltoniam_with_beta_and_mass}, we denote by
    \begin{equation}
    \label{eq:partition_function_after_sources}
    Z_{\Lambda}^{\sf YMH,s} = \int_{[-\pi,\pi]^{E(\Lambda)}} \exp\Big(-\beta \, \mathcal{H}_{\Lambda,{\sf s}}(\boldsymbol{\theta})\Big) \prod_{e\in E(\Lambda)}  {\rm d}\theta_e
    \end{equation}
    the new modified partition function. We note that all bounds obtained throughout are uniform in that parameters $t_x$ and $t_y$. Hence, to lighten on the notation, we will usually omit the dependence on these parameters. In particular, for the case $t_x=t_y=0$ which corresponds to the original partition function~\eqref{eq:intro-def_of_YMH_measure}, everything we derive below is valid. The only places where the ``sources" play any role is in the actual proof of Theorem~\ref{thm:exponential_clustering_in_lattice_YMH} (given in Section~\ref{sec:convergence_of_cluster_expansion} below) and in the bounds obtained in Section~\ref{sec:dealing_with_non_gaussian_correction}, where we need to take the function $V_{\sf s}$ into account in our derivation. Obviously
    \begin{equation}
    \label{eq:diffrentiating_the_log_partition_function}
     \frac{\partial}{\partial t_x} \, \frac{\partial}{\partial t_x} \bigg|_{t_x=t_y=0} \log Z_{\Lambda}^{\sf YMH,s} =    \bE_{\Lambda}^{\sf YMH}\big[\theta_x \theta_y\big] - \bE_{\Lambda}^{\sf YMH}\big[\theta_x \big]\cdot \bE_{\Lambda}^{\sf YMH}\big[ \theta_y\big] \, ,
    \end{equation}
    which is the reason for considering the extra ``sources" term~\eqref{eq:def_of_V_s}. We would like to have a convenient representation for $\log Z_{\Lambda}^{\sf YMH,s}$, which will allow through~\eqref{eq:diffrentiating_the_log_partition_function} to obtain the desired bound for the covariance. The goal of this section is to find such a representation, given by the next proposition. In words, we obtain a perturbative expansion for the the Yang--Mills--Higgs partition function around a particular Gaussian partition function.
    \begin{proposition}
    \label{proposition:expansion_for_YMH_partition_function_before_log}
        Let $Z_{\Lambda}^{{\sf YMH,s}}$ be the partition function~\eqref{eq:partition_function_after_sources}. Then
        \begin{equation}
        \label{eq:polymer_expansion_for_YMH_partition_function_before_log}
            Z_{\Lambda}^{{\sf YMH,s}} = Z_{\Lambda,f}^{{\sf G}}\cdot \Xi_{\Lambda}(\mathbf{0})\cdot \bigg( \sum_{n\ge 1} \, \sum_{\substack{{\sf P}_1,\ldots, {\sf P}_n\subset \La^\prime \\ {\sf P}_i \, \text{\normalfont connected}}} \,  \prod_{j=1}^{n} {\sf w}({\sf P}_{j})  \cdot \prod_{i<j} \delta({\sf P}_i,{\sf P}_j)  \bigg)\, ,
        \end{equation}
        where $Z_{\Lambda,f}^{\sf G}$ is a Gaussian partition function given by Definition~\ref{def:lattice_proca_field} below, $\Xi_{\Lambda}(\mathbf{0})$ is given by equation~\eqref{eq:def_of_Xi_with_interpolation_s},  ${\sf w}$ is given by equation~\eqref{eq:def_of_w}, $\Lambda'$ is the coarse lattice defined via the equation~\eqref{eq:def_of_la_prime}, and
        \begin{equation}
        \label{eq:def_of_delta_interaction}
        \delta({\sf P}_i,{\sf P}_j) = \begin{cases}
            0 & \text{\normalfont if ${\sf P}_i\cup {\sf P}_j$ is connected}, \\ 1 & \text{\normalfont otherwise.} 
        \end{cases}    
        \end{equation}
    \end{proposition}
    \noindent
    Note that the sum in~\eqref{eq:polymer_expansion_for_YMH_partition_function_before_log} is in fact finite. An expansion of the form~\eqref{eq:polymer_expansion_for_YMH_partition_function_before_log} is commonly referred to in the literature as a \emph{polymer expansion}; see for instance~\cite[Chapter~5]{Friedli-Velenik_book}. Borrowing the terminology from~\cite{Friedli-Velenik_book}, the polymers are precisely connected subsets ${\sf P}\subset \La^\prime $ and the sum~\eqref{eq:polymer_expansion_for_YMH_partition_function_before_log} is of the form~\cite[Definition~5.2]{Friedli-Velenik_book}. We also remark that~\eqref{eq:def_of_delta_interaction} is an example of hard-core interactions between polymers, see~\cite[Section~5.6]{Friedli-Velenik_book}. 
    While it is not so easy to grasp Proposition~\ref{proposition:expansion_for_YMH_partition_function_before_log} just yet (as most notation there are yet to be defined), we devote this section to systematically introduce all notation which appear in Proposition~\ref{proposition:expansion_for_YMH_partition_function_before_log}, and show how the polymer expansion for $Z_{\Lambda}^{\sf YMH, s}$ is derived. 
    
    \subsection{The Proca field}
    As we already mentioned above, the basic idea of our analysis is that for $\beta$ large, the Yang--Mills--Higgs measure~\eqref{eq:intro-def_of_YMH_measure} is a small perturbation of a particular massive Gaussian field, known as the (Euclidean) lattice Proca field. 
    \begin{definition}
        \label{def:lattice_proca_field}
        Let $\Lambda\subset \bZ^d$ be a finite domain and let $m>0$. The (Gaussian) Proca action on $\Lambda$ is the function $S_{\Lambda} : \bR^{E(\Lambda)} \to \bR_{\ge 0}$ given by
        \begin{equation}
        \label{eq:def_of_gaussian_action}
            S_{\Lambda} (\boldsymbol{\theta}) = \frac{1}{2} \sum_{p\in P(\Lambda)} \big({\sf d}\theta_p\big)^2 + \frac{m}{2}\sum_{e\in E(\Lambda)} (\theta_e)^2 \, .    
        \end{equation}
        For $\eta:\partial \Lambda \to \bR$  and $\beta>0$, the lattice Proca field on $\Lambda$ with boundary conditions $\eta$ is the Gibbs measure on $\boldsymbol{\theta} = (\theta_e)_{e\in E(\Lambda)} \subset \bR^{E(\Lambda)}$ given by 
        \begin{equation}
        \label{eq:def_of_lattice_proca_field}
        {\rm d}\mu_{\Lambda,\eta}^{\sf G}(\boldsymbol{\theta}) = \frac{1}{Z_{\Lambda,\eta}^{\sf G}} \exp\Big( - \beta S_{\Lambda}(\boldsymbol{\theta})\Big) \prod_{e\in E(\Lambda) \setminus \partial \Lambda} {\rm d} \theta_e \prod_{e\in \partial \Lambda} \delta_{\eta_e}(\theta_e) \, .    
        \end{equation}
        Similarly, we let $\mu_{\Lambda,{f}}^{\sf G}$ and $Z_{\Lambda,f}^{\sf G}$ to be the Gibbs measure and the partition function of the lattice Proca field with free boundary conditions, defined exactly as in~\eqref{eq:def_of_lattice_proca_field} but with no restrictions on the boundary. We denote by $\bE_{\La,\eta}^{\sf G}[\cdot]$ and $\bE_{\La,f}^{\sf G}[\cdot ]$ the corresponding expectation operators. 
    \end{definition}
    We shall see below (Lemma~\ref{lemma:proca_field_is_massive}) that the lattice Proca field is indeed massive, i.e., has exponential decay of correlations. Expanding around a given Gaussian field in the large $\beta$ regime is a common technique in the study of Yang--Mills--Higgs measures, see~\cite{Balaban-Imbrie-Jaffe-Brydges, Frohlich-Spencer, Glimm-Jaffe-book, Glimm-Jaffe-Spencer, Osterwalder-Seiler}. We shall employ a similar strategy. As we already mentioned, our treatment here is in large part borrowed from~\cite{Balaban-Imbrie-Jaffe-Brydges}, which in turn is borrowed from~\cite{Glimm-Jaffe-Spencer}. Set
    \begin{equation*}
        g(t) = 1- \cos(t) - \frac{t^2}{2} \, ,
    \end{equation*}
    and let
    \begin{equation}
        \label{eq:def_of_V}
        V_{\Lambda}(\boldsymbol\theta) = \sum_{p\in P(\Lambda)} g({\sf d}\theta_p) + m \sum_{e\in E(\Lambda)} g(\theta_e)  + V_{{\sf s}}(\boldsymbol\theta)\, , 
    \end{equation}
    with $V_{{\sf s}}$ given by~\eqref{eq:def_of_V_s}.
    The Yang--Mills--Higgs density~\eqref{eq:intro-def_of_YMH_measure} can formally be written as
    \begin{align}
    \label{eq:heuristic_gaussian_approximation}
    \nonumber 
    {\rm d}\mu_{\Lambda}^{\sf YMH} (\boldsymbol{\theta}) &\propto e^{-\beta V_{\Lambda}(\boldsymbol{\theta}) - \beta S_{\Lambda}(\boldsymbol{\theta})} \prod_{e\in E(\Lambda)} \one_{\{|\theta_e|\le \pi \}}{\rm d}\theta_e \\ &\propto e^{-\beta V_{\Lambda}(\boldsymbol{\theta})}\cdot \Big(\prod_{e\in E(\Lambda)} \one_{\{|\theta_e|\le \pi \}} \Big) \, {\rm d}\mu_{\Lambda,f}^{\sf G}(\boldsymbol{\theta})  \, .
    \end{align}
    Heuristically, both $e^{-\beta V_{\Lambda}}$ and the product over indicators are expected to be (typically) close to 1, since most field elements sampled from the lattice Proca field have
    \[
    |\theta_e| \approx m^{-1/2} \beta^{-1/2}\, ,
    \]
    and in that range we have
    \[
    g(\theta_e) \lesssim (\theta_e)^{4} \lesssim m^{-2} \beta^{-2} \, .
    \]
    Therefore, $\beta g(\theta_e)$ is expected to be small for most field elements, assuming that $\beta$ is large enough. 
    
    \subsection{Partition of unity for field values}\label{subsection:partition_of_unity}
    To make the above reasoning rigorous, we first need to remove from consideration edges where the Yang--Mills--Higgs field is atypically large. Letting 
    \begin{equation}
        \label{eq:def_of_T_beta}
        T_\beta = \frac{\log^{d+2}(\beta)}{\sqrt{\beta}} \, ,
    \end{equation}
    we say that a field element $\theta_e$, $e\in E(\Lambda)$ is \emph{large} if $|\theta_e| \ge T_\beta$, and otherwise we say that the field element is \emph{small}. 
    \begin{definition}
        \label{def:good_and_bad_blocks}
        We say that an $L$-block $v\in \La^\prime$ is \emph{good} if all of its associated field elements $\theta_e$ from $e\in \cE\big(Q(v)\big)$ are small. 
        Otherwise, we say that the $L$-block is \emph{bad}. We denote by $\cB\subset \La^\prime$ the collection of all bad $L$-blocks, and by $\mathcal{G} = \La^\prime\setminus \cB$ the collection of all good $L$-blocks. 
    \end{definition}
    \noindent
    We note that according to Definition~\ref{def:good_and_bad_blocks}, an $L$-block is declared bad if there is a large field element in its interior or its boundary. The heuristic Gaussian approximation~\eqref{eq:heuristic_gaussian_approximation}  is only valid for small field elements, so our goal now is to first isolate the region $\cB$ where the field might be large. For technical reasons that will become clear later on, we will remove this region in a ``smooth" way. Indeed, let $\chi:\bR\to [0,1]$ be a $C^\infty$ smooth function, such that\footnote{To see that such a function exists, simply take the standard bump-function $\eta(t) = e^{-1/(1-t^2)} \, \one_{\{|t|\le 1\}}$ and set $$\chi(t) = \frac{\eta(t/2)}{\eta(t/2) + \eta(1-|t|/2)}\, .$$ We remark that this is a specific example of a function of Gevrey class~\cite{Hormander-PDEI}, and that the constant $2$ in item (3) cannot be improved to 1, as compact support and real-analyticity cannot co-exist.}:
    \begin{enumerate}
        \item[(1)] $\chi(t) = \chi(-t)$ for all $t\in \bR$;
        \item[(2)] $\chi(t) = 1$ if $|t|\le 1$ and $\chi(t) = 0$ if $|t|\ge 2$; 
        \item[(3)] For all $k\ge 1$, we have $\displaystyle \sup_{t\in \bR} |\chi^{(k)}(t)| \le C (k!)^{2}$ for some $C>0$. 
    \end{enumerate}
    Setting $\zeta(t) = 1- \chi(t)$, we can write 
    \begin{align*}
    \nonumber
        1 &= \prod_{e\in E(\Lambda)} \bigg(\chi\Big(\frac{|\theta_e|}{T_\beta}\Big) + \zeta\Big(\frac{|\theta_e|}{T_\beta}\Big) \bigg) \\ &= \sum_{\mathcal{L} \subset E(\Lambda)} \bigg( \prod_{e\in E(\La)\setminus \mathcal{L} } \chi\Big(\frac{|\theta_e|}{T_\beta}\Big) \cdot  \prod_{e\in \mathcal{L}}\zeta\Big(\frac{|\theta_e|}{T_\beta}\Big)\bigg) \, .
    \end{align*}
    In the above sum, the summand is zero unless we have $\mathcal{L}\subset\cE (\cB)$, where $\cB$ is the collection of bad $L$-blocks, as given by Definition~\ref{def:good_and_bad_blocks}. Therefore, we can write 
    \begin{equation}   \label{eq:partition_of_unity_into_large_and_small_fields}
        1  = \sum_{\cB\subset \La^\prime} \chi_{\cG} \cdot \zeta_{\cB} \, , 
    \end{equation}
    where
    \begin{equation} \label{eq:def_of_chi_G}
        \chi_{\cG} = \prod_{v\in \cG} \Big( \prod_{e\in E(Q(v))} \chi\Big(\frac{|\theta_e|}{T_\beta}\Big) \Big) \, ,
    \end{equation}
    and
    \begin{equation} \label{eq:def_of_zeta_B}
        \zeta_{\cB} = \mathbf{1}_{\{\cB \ \text{is the collection of all bad $L$-blocks}\}} \cdot \sum_{\mathcal{L} \subset \cE(\cB)} \bigg(\prod_{e\in \mathcal{L}} \zeta\Big(\frac{|\theta_e|}{T_\beta}\Big) \cdot \prod_{\substack{e\in E(\La) \setminus \mathcal{L} \\ e \in \cE(\cB) }} \chi\Big(\frac{|\theta_e|}{T_\beta}\Big)\bigg)  \, .
    \end{equation}
    The edges in $\mathcal{G}$ are those where the Gaussian approximation will be valid. We denote by $S_{\Lambda,\mathcal{G}}$ the Gaussian action~\eqref{eq:def_of_gaussian_action}, after we throw away all interactions coming from edges and plaquettes which do not intersect $\mathcal{G}$, that is 
    \begin{equation}
        \label{eq:gaussian_action_in_G}
        S_{\Lambda,\mathcal{G}}(\boldsymbol{\theta}) = \frac{1}{2} \sum_{\substack{p\in P(\Lambda) \\ E(p)\cap \cE (\mathcal{G}) \not= \emptyset}}  \big({\sf d}\theta_p\big)^2 + \frac{m}{2}\sum_{e\in \cE (\mathcal{G})} (\theta_e)^2 \, . 
    \end{equation}
    We will also denote by $V_{\Lambda,\mathcal{G}}$ the action~\eqref{eq:def_of_V} after we threw away interactions coming from edges and plaquettes which do not intersect $\mathcal{G}$, defined in a similar way as~\eqref{eq:gaussian_action_in_G}. Setting
    \begin{equation}\label{eq:action_restricted_to_H}
        S_{\Lambda}^{\mathcal{B}}(\boldsymbol{\theta}) = S_{\Lambda}(\boldsymbol{\theta}) - S_{\Lambda,\mathcal{G}}(\boldsymbol{\theta}) \, , \quad \text{and}\quad V_{\Lambda}^{\mathcal{B}}(\boldsymbol{\theta}) = V_{\Lambda}(\boldsymbol{\theta}) - V_{\Lambda,\mathcal{G}}(\boldsymbol{\theta}) \, ,
    \end{equation}
    the partition of unity~\eqref{eq:partition_of_unity_into_large_and_small_fields} implies that 
    \begin{align*}
          Z_{\Lambda}^{\sf YMH,s} &= \int_{[-\pi,\pi]^{E(\Lambda)}} \exp\Big(-\beta \, \mathcal{H}_{\Lambda,{\sf s}}(\boldsymbol{\theta})\Big) \prod_{e\in E(\Lambda)}  {\rm d}\theta_e   \\ &= \sum_{\cB\subset \La^\prime} \int_{[-\pi,\pi]^{\cE(\cB)}} e^{-\beta V_{\Lambda}^{\mathcal{B}}(\boldsymbol{\eta})- \beta S_{\Lambda}^{\mathcal{B}}(\boldsymbol{\eta})} \cdot \zeta_{\mathcal{B}} \cdot \bigg(\int_{\bR^{\cE (\mathcal{G})}} e^{-\beta V_{\Lambda,\mathcal{G}}(\boldsymbol{\theta})- \beta S_{\Lambda,\mathcal{G}}(\boldsymbol{\theta})} \cdot \chi_{\mathcal{G}} \, {\rm d}\boldsymbol{\theta}_{\mathcal{G}}\bigg) \, {\rm d}\boldsymbol{\eta}_{\mathcal{B}} \, , 
    \end{align*}
    where ${\rm d}\boldsymbol{\theta}_{\mathcal{G}} = \prod_{e\in \cE (\mathcal{G})} {\rm d} \theta_e$. We observe that $S_{\Lambda,\mathcal{G}}$ is exactly the Proca field action from Definition~\ref{def:lattice_proca_field}, restricted to the edges contained in $\cE(\mathcal{G})$. Hence,
    \[
    \int_{\bR^{\cE (\mathcal{G})}} e^{-\beta V_{\Lambda,\mathcal{G}}(\boldsymbol{\theta})- \beta S_{\Lambda,\mathcal{G}}(\boldsymbol{\theta})} \cdot \chi_{\mathcal{G}} \, {\rm d}\boldsymbol{\theta}_{\mathcal{G}} = Z_{\mathcal{G},\eta}^{\sf G} \cdot \bE_{\mathcal{G},\eta}^{\sf G}\big[e^{-\beta V_{\Lambda,\mathcal{G}}(\boldsymbol{\theta})} \chi_{\mathcal{G}}\big] \, , 
    \]
    where $\eta$ are the boundary conditions imposed by the configuration in $\mathcal{B}$. Denoting by
    \begin{equation}
        \label{eq:def_of_Xi_G_eta}
        \Xi_{\mathcal{G},\eta} = \bE_{\mathcal{G},\eta}^{\sf G}\big[e^{-\beta V_{\Lambda,\mathcal{G}}(\boldsymbol{\theta})} \chi_{\mathcal{G}}\big]  \, ,
    \end{equation}
    we arrive at the formula
    \begin{equation}
        \label{eq:YMH_partition_function_first_reduction}
        Z_{\Lambda}^{\sf YMH,s} = \sum_{\mathcal{B}\subset \La^\prime } \int_{[-\pi,\pi]^{\cE(\cB)}} e^{-\beta V_{\Lambda}^{\mathcal{B}}(\boldsymbol{\eta})- \beta S_{\Lambda}^{\mathcal{B}}(\boldsymbol{\eta})} \cdot \zeta_{\mathcal{B}} \cdot Z_{\mathcal{G},\eta}^{\sf G} \cdot \Xi_{\mathcal{G},\eta}\,  {\rm d}\boldsymbol{\eta}_{\mathcal{B}} \, .
    \end{equation}
    The equality~\eqref{eq:YMH_partition_function_first_reduction} is the starting point of the proof of Proposition~\ref{proposition:expansion_for_YMH_partition_function_before_log}. From this point, the analysis roughly consists of three steps: 
    \begin{enumerate}
        \item Deriving an expansion for the Gaussian partition function $Z_{\mathcal{G},\eta}^{\sf G}$;
        \item Deriving a similar expansion for $\Xi_{\mathcal{G},\eta}$;
        \item Combining both expansions with the sum over $\mathcal{B}$ to arrive at one uniting expansion. 
    \end{enumerate}
    We note that uniting several different expansions into one compatible polymer system, as we preform below, is a recurring theme in constructive statistical mechanics; see for example Brydges's lecture notes~\cite[page~157]{Brydges-course} for a general discussion on this point.
    
    \subsection{Expansion for the Gaussian partition function} \label{subsection:expansion_of_gaussian_partition}
    Before explaining the expansion for the Gaussian term, we introduce some notation. Let 
    \begin{equation*}
    r_\beta = \log^2\beta\, ,    
    \end{equation*}
    and recall the that $\cB$ and $\cG = \La^\prime \setminus \cB $ are given by Definition~\ref{def:good_and_bad_blocks}. We denote by 
    \begin{equation}
        \label{eq:def_of_B_1}
        \mathcal{B}_1 = \big\{ v\in \La^\prime \, : \, \text{dist}(v, \mathcal{B}) \le r_\beta \big\} 
        \, .
    \end{equation}
    In words, $\mathcal{B}_1$ is the $r_{\beta}$-neighborhood of $\mathcal{B}$ in the coarse lattice $L\Z^d$. We further denote by 
    $$\mathcal{G}_1 = \La^\prime \setminus \mathcal{B}_1$$ and note that $\mathcal{G}_1 \subset \mathcal{G}$. The idea is that in $\mathcal{G}_1$ the boundary effect imposed by the field values in $\mathcal{B}$ already has a negligible effect, while since $r_\beta^{d} = o(\beta T_\beta^2)$ the entropy gain by considering $\mathcal{B}_1$ instead of $\mathcal{B}$ is also not significant. 
    To expand the Gaussian partition function $Z_{\mathcal{G},\eta}^{\sf G}$, we shall appeal to the Glimm--Jaffe--Spencer expansion~\cite{Glimm-Jaffe-book,Glimm-Jaffe-Spencer}. The basic idea is to interpolate between the original Gaussian Proca field and another Gaussian field where all $L$-blocks are independent. Recall that $P^\prime(\mathcal{G}_1)$ is the set of all faces between $L$-blocks in $\mathcal{G}_1$. We will consider a collection of parameters $\mathbf{s} = (s_{\sf F})_{{\sf F} \in P^\prime(\mathcal{G}_1)}$, all in $[0,1]$, and define a new Gaussian action via
    \begin{equation}
        \label{eq:def_of_gaussian_action_with_s}
        S_{\Lambda,\mathcal{G}}(\boldsymbol{\theta};\mathbf{s}) = \frac{1}{2} \sum_{\substack{p\in P(\Lambda) \\ E(p)\cap \cE(\mathcal{G}) \not= \emptyset}}  \sigma_p(\mathbf{s}) \big({\sf d}\theta_p\big)^2 + \frac{m}{2}\sum_{e\in \cE(\mathcal{G})} (\theta_e)^2 \, , 
    \end{equation}
    where,
    \begin{equation}
        \label{eq:def_of_sigma_p}
        \sigma_p(\mathbf{s}) = \begin{cases}
            s_{\sf F} & \text{if} \ p \ \text{contains an edge from the face} \  {\sf F} \in P^\prime(\mathcal{G}_1) \, , \\ 1 & \text{else.}
        \end{cases}
    \end{equation}
    There are some plaquettes that contain edges from multiple faces, namely, the plaquettes touching the corners of the $L$-blocks. For these plaquettes, it does not really matter how we choose $s_{\sf F}$, as long as we do so consistently. For concreteness, here we will choose the minimal face with respect to the natural lexicographical order. Note that, when $s_{\sf F} \equiv 1$, that action~\eqref{eq:def_of_gaussian_action_with_s} reduces to~\eqref{eq:def_of_gaussian_action}. Similarly as in Definition~\ref{def:lattice_proca_field}, we will denote 
    \begin{equation*}
        \label{eq:def_of_gaussian_mesure_with_s}
        {\rm d}\mu_{\mathcal{G},\eta,\mathbf{s}}^{\sf G}(\boldsymbol{\theta}) = \frac{1}{Z_{\mathcal{G},\eta,\mathbf{s}}^{\sf G}} \exp\Big( - \beta S_{\Lambda,\mathcal{G}}(\boldsymbol{\theta};\mathbf{s})\Big) \prod_{e\in \cE(\mathcal{G}) \setminus \partial \cE(\mathcal{G}) } {\rm d} \theta_e \prod_{e\in \partial \cE(\mathcal{G})} \delta_{\eta_e}(\theta_e) \, .    
    \end{equation*}   
    the probability measure (and partition function) which corresponds to~\eqref{eq:def_of_gaussian_action_with_s}. Here $\partial \cE(\mathcal{G})$ is the set of all boundary edges in $\cE(\cG)$, that is, the edges in a face connecting a block from $\cG$ to a block from $\cB$. 
    When $s_p \equiv 0$, the field values in different $L$-blocks in $\mathcal{G}_1$ are independent. We will sometimes also denote the relevant partition function as $Z_{\mathcal{G},\eta,\mathbf{s}}^{\sf G} = Z_{\mathcal{G},\eta}^{\sf G}(\mathbf{s})$, depending on how much we want to emphasis the parameters $\mathbf{s}$. 
    \begin{claim}
        \label{claim:fund_thm_of_calculus}
        Let $F: [0,1]^{S} \to \bR$ be a smooth function, where $S$ is some finite set. Then
        \begin{equation}\label{eq:fund_thm_of_calculus}
          F(\mathbf{1}) = \sum_{\Ga\subset S} \Big(\int_{[0,1]^{\Ga}} \partial^\Ga  F(\mathbf{s}_{\Ga}) \, {\rm d} \mathbf{s}_{\Ga} \Big) \, ,
        \end{equation}
        where $\Ga$ is summed over all subsets of $S$ (including the null set) and 
        \[
        \partial^\Ga = \prod_{\gamma\in \Ga} \frac{\partial}{\partial s_\gamma} \,, 
        \qquad (\mathbf{s}_\Ga)_\gamma = \begin{cases}
            s_\gamma & \gamma\in \Ga \, , \\  0 & \text{\normalfont otherwise.}
        \end{cases} 
        \]
        When $\Ga = \emptyset$, the derivative and integral in~\eqref{eq:fund_thm_of_calculus} are omitted. 
    \end{claim}
    \begin{proof}
        We proceed by induction on the size of $S$. When $|S|=1$, the equality~\eqref{eq:fund_thm_of_calculus} is nothing but the fundamental theorem of calculus. For the induction step, we write $S = S^\prime \cup \{s\}$ and get that
    \[
     F(\mathbf{1},s) \stackrel{\eqref{eq:fund_thm_of_calculus}}{=} \sum_{\Ga\subset S^\prime} \Big(\int_{[0,1]^{\Ga}} \partial^\Ga  F(\mathbf{s}_{\Ga},s) \, {\rm d} \mathbf{s}_{\Ga} \Big) \, .
    \]
    Another application of the fundamental theorem of calculus, applied this time with the variable $s\in[0,1]$, completes the proof of the claim. 
    \end{proof}
    Applying Claim~\ref{claim:fund_thm_of_calculus} with $S = P^\prime(\mathcal{G}_1)$ and $F(\mathbf{s}) = \log Z_{\mathcal{G},\eta}^{\sf G}(\mathbf{s})$, we see that
    \begin{equation}
        \label{eq:gaussian_partition_function_interpolation_using_fund_thm}
        \log Z_{\mathcal{G},\eta}^{\sf G} = \log Z_{\mathcal{G},\eta,\mathbf{1}}^{\sf G} = \sum_{\Gamma \subset P^\prime(\mathcal{G}_1)} W_1(\Ga;\mathcal{G}) \, ,
    \end{equation}
    where $\Ga$ is summed over all subsets of faces in $P^\prime(\mathcal{G}_1)$ and
    \begin{equation}
        \label{eq:def_of_W_1}
        W_1(\Ga;\mathcal{G}) = \int_{[0,1]^{\Gamma}} \partial^\Gamma \log   Z_{\mathcal{G},\eta}^{\sf G}(\mathbf{s}_{\Ga}) \,  {\rm d} s_{\Ga} \, .
    \end{equation}
    Repeating the same expansion for $\mathcal{G} = \La^\prime$ (in which case $\mathcal{B} = \emptyset$ and there are no boundary conditions), we conclude from~\eqref{eq:gaussian_partition_function_interpolation_using_fund_thm} that 
    \begin{equation}
    \label{eq:ratio_of_gaussian_partition_functions_after_expansion}
        Z_{\mathcal{G},\eta}^{\sf G} = Z_{\Lambda,f}^{\sf G}\cdot \frac{Z_{\mathcal{G},\eta}^{\sf G}(\mathbf{0})}{Z_{\Lambda,f}^{\sf G}(\mathbf{0})} \cdot \exp\Big( \sum_{\Ga\not=\emptyset} W_2(\Ga;\mathcal{G}) \Big)
    \end{equation}
    where
    \begin{equation}
        \label{eq:def_of_W_2}
        W_2(\Ga;\mathcal{G}) = W_1(\Ga;\mathcal{G}) - W_1(\Ga;\Lambda^\prime) \, .
    \end{equation}
    Plugging~\eqref{eq:ratio_of_gaussian_partition_functions_after_expansion} into~\eqref{eq:YMH_partition_function_first_reduction} we arrive at
    \begin{align}
        \label{eq:YMH_partition_function_second_reduction}
        Z_{\Lambda}^{\sf YMH,s}  &= Z_{\Lambda,f}^{\sf G} \cdot \sum_{\mathcal{B}\subset E(\Lambda)} \int_{[-\pi,\pi]^{\cE(\mathcal{B})}} e^{-\beta V_{\Lambda}^{\mathcal{B}}(\boldsymbol{\eta})- \beta S_{\Lambda}^{\mathcal{B}}(\boldsymbol{\eta})} \notag \\
        &\hskip1in \cdot \zeta_{\mathcal{B}} \cdot \frac{Z_{\mathcal{G},\eta}^{\sf G}(\mathbf{0})}{Z_{\Lambda,f}^{\sf G}(\mathbf{0})} \cdot \exp\Big( \sum_{\Ga\not=\emptyset} W_2(\Ga;\mathcal{G}) \Big) \cdot \Xi_{\mathcal{G},\eta}\,  {\rm d}\boldsymbol{\eta}_{\mathcal{B}} \, .
    \end{align}
    The next simple claim about the weights $W_2$ will be used several times in what follows, so we document it here for future reference. 
    \begin{claim}
        \label{claim:condition_for_which_W_2_non_zero}
        We have $W_2(\Ga,\cG) = 0$ if either:
        \begin{enumerate}
            \item $\Ga$ is not connected; or
            \item $X(\Ga) \cup \cB_1$ is not connected (i.e., $\Ga$ does not touch $\cB_1$), where $\cB_1$ is given by~\eqref{eq:def_of_B_1}. 
        \end{enumerate}
    \end{claim}
    \begin{proof}
        If $\Ga$ is not connected, then $Z_{\Lambda,\eta}^{\sf G}(\mathbf{s}_\Ga)$ factors into contributions from its connected components. Therefore, when applying $\partial^\Ga$ to $\log Z_{\Lambda,\eta}^{\sf G}(\mathbf{s}_\Ga)$ each component is necessarily differentiated by a variable not present in the component, and in view of~\eqref{eq:def_of_W_1} we have $W_1(\Ga;\cG) = 0$. Furthermore, if $\Ga$ does not touch $\cB_1$, it is apparent from~\eqref{eq:def_of_W_1} that
        \[
        W_1(\Ga; \cG) = W_1(\Ga; \La^\prime) \, ,
        \]
        and hence $W_2(\Gamma,\cG) = 0$ in this case as well.
    \end{proof}
    \subsection{Expansion for non-Gaussian correction} \label{subsection:expansion_for_non_gaussian_correction}
    We derive a similar expansion for the term 
    \[
    \Xi_{\mathcal{G},\eta} = \bE_{\mathcal{G},\eta}^{\sf G}\big[e^{-\beta V_{\Lambda,\mathcal{G}}(\boldsymbol{\theta})} \chi_{\mathcal{G}}\big] = \int_{\bR^{\cE(\mathcal{G})}}e^{-\beta V_{\Lambda,\mathcal{G}}(\boldsymbol{\theta})} \chi_{\mathcal{G}}(\boldsymbol{\theta}) \, {\rm d}\mu_{\mathcal{G},\eta}^{\sf G} \, .
    \]
    Recall that $\mu_{\Lambda,\eta,\mathbf{s}}^{\sf G}$ is the probability measure associated with the Gaussian action~\eqref{eq:def_of_gaussian_action_with_s}. In a similar fashion, we can interpolate the action $V_{\La,\mathcal{G}}$ as
    \begin{equation}
        \label{eq:def_of_V_interpolation_s}
        V_{\Lambda,\mathcal{G}}(\boldsymbol\theta;\mathbf{s}) = \sum_{\substack{p\in P(\Lambda) \\ E(p)\cap \cE(\mathcal{G}) \not= \emptyset}} \sigma_p(\mathbf{s})  g({\sf d}\theta_p) + m \sum_{e\in E(\mathcal{G})} g(\theta_e)  + V_{{\sf s}}(\boldsymbol\theta)\, , 
    \end{equation}
    where $\sigma_p(\mathbf{s})$ is given by~\eqref{eq:def_of_sigma_p}. We remark that without loss of generality, we may assume that the edges $x,y\in E(\Lambda)$ from the statement of Theorem~\ref{thm:exponential_clustering_in_lattice_YMH} are not contained in faces from $P^\prime(\mathcal{G}_1)$, as a slight modification of the parameter $L$ will guarantee this is not the case. Hence, the interpolation does not effect the term $V_{\sf s}$ in~\eqref{eq:def_of_V_interpolation_s}. Claim~\ref{claim:fund_thm_of_calculus} applied with
    \begin{equation}
    \label{eq:def_of_Xi_with_interpolation_s}
        \Xi_{\mathcal{G},\eta}(\mathbf{s}) = \int_{\bR^{\cE(\mathcal{G})}}e^{-\beta V_{\Lambda,\mathcal{G}}(\boldsymbol{\theta};\mathbf{s})} \chi_{\mathcal{G}}(\boldsymbol{\theta}) \, {\rm d}\mu_{\mathcal{G},\eta,\mathbf{s}}^{\sf G}  \, ,
    \end{equation}
    gives that
    \begin{equation}
    \label{eq:Xi_after_fund_thm_of_calculus}
        \Xi_{\mathcal{G},\eta} = \Xi_{\mathcal{G},\eta}(\mathbf{1}) = \sum_{\Delta\subset P^\prime(\mathcal{G}_1)} \int_{[0,1]^\Delta} \partial^\Delta \Xi_{\mathcal{G},\eta}(\mathbf{s}_{\Delta}) \, {\rm d} \mathbf{s}_\Delta \, .
    \end{equation}
    With the notation above, $\Xi_\mathcal{G,\eta}(\mathbf{0})$ factors into contributions from the different $L$-blocks in $\mathcal{G}_1$, and then also the contribution from the field in $\mathcal{G}\setminus \mathcal{G}_1$. For technical reasons that will be clear later on, we define $\Xi_{\mathcal{G},\eta}(\mathbf{0},+)$ to be the corresponding expectation when we also decouple $L$-blocks from $\mathcal{G}\setminus \mathcal{G}_1$. More formally, recall that $P^\prime(\mathcal{G})$ is the set of all faces in $\mathcal{G}$ between its different $L$-blocks. We set 
    \begin{equation}
        \label{eq:def_of_gaussian_action_with_0_+}
        S_{\Lambda,\mathcal{G}}(\boldsymbol{\theta};\mathbf{0},+) = \frac{1}{2} \sum_{\substack{p\in P(\Lambda) \\ E(p)\cap \cE(\mathcal{G})\not= \emptyset}}  \widetilde\sigma_p \big({\sf d}\theta_p\big)^2 + \frac{m}{2}\sum_{e\in \cE(\mathcal{G})} (\theta_e)^2 \, , 
    \end{equation}
    with
    \[
    \widetilde \sigma_p = \begin{cases}
        0 & \text{if $p$ contains an edge from a face }{\sf F} \in P^\prime(\mathcal{G}) \, ,  \\ 1 & \text{otherwise.}
    \end{cases}
    \]
    Let $\mu_{\mathcal{G},\eta,\mathbf{0},+}^{\sf G}$ be the probability measure corresponding to the Gaussian action~\eqref{eq:def_of_gaussian_action_with_0_+}. Similarly, we set 
        \begin{equation*}
        V_{\Lambda,\cG}(\boldsymbol\theta;\mathbf{0},+) = \sum_{\substack{p\in P(\Lambda) \\ E(p)\cap \cE(\mathcal{G}) \not= \emptyset}} \widetilde \sigma_p \,   g({\sf d}\theta_p) + m \sum_{e\in \cE(\cG)} g(\theta_e)  + V_{{\sf s}}(\boldsymbol\theta)\, , 
    \end{equation*}
    and so
    \begin{equation}
    \label{eq:def_of_Xi_0_+}
        \Xi_{\mathcal{G},\eta}(\mathbf{0},+) = \int_{\bR^{\mathcal{G}}} e^{-\beta V_{\Lambda,\mathcal{G}}(\boldsymbol{\theta};\mathbf{0},+)} \, \chi_{\mathcal{G}}(\boldsymbol{\theta}) \, {\rm d}\mu_{\mathcal{G},\eta,\mathbf{0},+}^{\sf G} \, .
    \end{equation}
    As we shall see below (see Claim~\ref{claim:two_sided_inequality_for_ratio_of_xi}), the ratio of $\Xi_{\mathcal{G},\eta}(\mathbf{0},+)$ and $\Xi_{\mathcal{G},\eta}(\mathbf{0})$ is typically quite close to one. By setting
    \begin{equation}
    \label{eq:def_of_K_1}
        K_1(\Delta;\mathcal{G}) = \frac{1}{\Xi_{\mathcal{G},\eta}(\mathbf{0},+)} \int_{[0,1]^\Delta} \partial^\Delta \Xi_{\mathcal{G},\eta}(\mathbf{s}_{\Delta}) \, {\rm d} \mathbf{s}_\Delta \, ,
    \end{equation}
    we see from~\eqref{eq:Xi_after_fund_thm_of_calculus} that
    \begin{equation*}
         \Xi_{\mathcal{G},\eta} = \Xi_{\mathcal{G},\eta}(\mathbf{0},+) \sum_{\Delta\subset P^\prime(\mathcal{G}_1)} K_1(\Delta;\mathcal{G}) \, . 
    \end{equation*}    
    Plugging this expansion into~\eqref{eq:YMH_partition_function_second_reduction}, we get that
    \begin{align}
        \label{eq:YMH_partition_function_third_reduction}
        Z_{\Lambda}^{\sf YMH,s}  &=   Z_{\Lambda,f}^{\sf G} \cdot \Xi_{\Lambda}(\mathbf{0}) \cdot \sum_{\mathcal{B}\subset \Lambda^\prime } \int_{[-\pi,\pi]^{\cE(\mathcal{B})}} \rho_{\mathcal{B}}(\boldsymbol{\eta}) \notag \\
        &\hskip1.5in \cdot   \exp\Big\{ \sum_{\Ga\not=\emptyset} W_2(\Ga;\mathcal{G}) \Big\} \cdot \Big(\sum_{\Delta} K_1(\Delta;\mathcal{G}) \Big)\,  {\rm d}\boldsymbol{\eta}_{\mathcal{B}} \, ,
    \end{align}
    where
    \begin{equation}
        \label{eq:def_of_rho_B}
        \rho_{\mathcal{B}}(\boldsymbol{\eta}) = e^{-\beta V_{\Lambda}^{\mathcal{B}}(\boldsymbol{\eta})- \beta S_{\Lambda}^{\mathcal{B}}(\boldsymbol{\eta})} \cdot \zeta_{\mathcal{B}}(\boldsymbol{\eta}) \cdot \frac{Z_{\mathcal{G},\eta}^{\sf G}(\mathbf{0})}{Z_{\Lambda,f}^{\sf G}(\mathbf{0})} \cdot \frac{\Xi_{\mathcal{G},\eta}(\mathbf{0},+)}{\Xi_{\Lambda}(\mathbf{0})} \, .
    \end{equation}
    \subsection{Uniting the expansions: Proof of Proposition~\ref{proposition:expansion_for_YMH_partition_function_before_log}}
    Looking at~\eqref{eq:YMH_partition_function_third_reduction}, we can expand the exponential as
    \begin{align} \label{eq:expanding_the_exponential_for_W_2} \nonumber 
        \exp\Big( \sum_{\Ga\not=\emptyset} W_2(\Ga;\mathcal{G}) \Big) &= \sum_{n\ge 0} \frac{1}{n!} \bigg(\sum_{\Ga_1\not= \emptyset} \cdots \sum_{\Ga_n \not= \emptyset} \,  \prod_{j=1}^n W_2(\Ga_j;\mathcal{G})  \bigg) \\ & = \sum_{\mathbf{\Ga} \in \mathbf{T}} \frac{1}{|\mathbf{\Ga}|!}  \prod_{\Ga\in \mathbf{\Ga}} W_2(\Ga;\mathcal{G}) \, ,
    \end{align}
    where $\mathbf{T}$ is the set of all ordered tuples, namely
    \begin{equation*}
        \mathbf{T} = \big\{ \boldsymbol{\Ga} = (\Ga_1,\ldots,\Ga_n) \, : \, \Ga_i\subset P^\prime(\cG_1) \text{ is non-empty and connected} \big\} \, .
    \end{equation*}
    Plugging~\eqref{eq:expanding_the_exponential_for_W_2} into~\eqref{eq:YMH_partition_function_third_reduction} we see that 
    \begin{align}
        \label{eq:YMH_partition_function_forth_reduction}
        Z_{\Lambda}^{\sf YMH,s}  &=  Z_{\Lambda,f}^{\sf G} \cdot \Xi_{\Lambda}(\mathbf{0}) \sum_{\mathcal{B}\subset \Lambda^\prime } \, \sum_{\mathbf{\Ga}\in \mathbf{T}} \,  \sum_{\Delta\subset P^(\cG_1)} \frac{1}{|\mathbf{\Ga}|!} \int_{[-\pi,\pi]^{\cE(\mathcal{B})}} \rho_{\mathcal{B}}(\boldsymbol{\eta}) \notag \\
        &\hskip2in \cdot   \prod_{\Ga\in \mathbf{\Ga}} W_2(\Ga;\mathcal{G}) \cdot K_1(\Delta;\mathcal{G}) \,  {\rm d}\boldsymbol{\eta}_{\mathcal{B}} \, .
    \end{align}
    To complete the proof of Proposition~\ref{proposition:expansion_for_YMH_partition_function_before_log}, we need one more definition.
    \begin{definition}
    \label{def:polymer}
    Given any connected ${\sf P} \subset \La^\prime$ we say that $(\mathcal{B},\mathbf{\Ga},\Delta)$ is a \emph{polymer} on ${\sf P}$ if
    \[
    {\sf P} = \cB \cup \Big(\bigcup_{\Ga\in \mathbf{\Ga}} X(\Ga)\Big) \cup X(\Delta) \, . 
    \]
    Here $X(\cdot)$ is defined via equation~\eqref{xsdef}. Furthermore, for ${\sf P}\subset \La^\prime$ we set 
    \begin{equation}
        \label{eq:def_of_w}
        {\sf w}({\sf P}) = \sum_{(\mathcal{B},\mathbf{\Ga},\Delta) \ \text{is a polymer on } {\sf P}} \frac{1}{|\mathbf{\Ga}|!} \int_{[-\pi,\pi]^{\cE(\mathcal{B})}} \rho_{\mathcal{B}}(\boldsymbol{\eta}) \cdot   \prod_{\Ga\in \mathbf{\Ga}} W_2(\Ga;\mathcal{G}) \cdot K_1(\Delta;\mathcal{G}) \,  {\rm d}\boldsymbol{\eta}_{\mathcal{B}}  
    \end{equation}
    if ${\sf P}$ is connected, and ${\sf w}(\sf P) = 0$ if ${\sf P}$ is not connected. 
    \end{definition}
    \noindent
    Note that not every configuration in the sum~\eqref{eq:YMH_partition_function_forth_reduction} is a polymer on some ${\sf P}$, as the resulting union may not be connected. We say that $\Delta \subset P^\prime(\mathcal{G}_1)$ is $\mathcal{G}_1$-connected if either $\Delta$ is connected or the connected components of $\Delta$ can be connected by a single path from $\mathcal{B}_1$. It is apparent from~\eqref{eq:def_of_K_1} that $K_1$ factors across $\mathcal{G}_1$-connected components; so we can write
    \begin{equation}\label{eq:K_1_factor_into_connected_components}
    K_1(\Delta;\mathcal{G}) = \prod_{\substack{\delta\subset \Delta \\ \delta \ \text{is a maximal  $\mathcal{G}_1$-connected component}}} K_1(\delta;\mathcal{G}) \, .
    \end{equation}
    Since the function $\rho_\cB$ factors across connected components of the bad $L$-blocks $\cB$ (and so does the integral in~\eqref{eq:YMH_partition_function_forth_reduction}), we can combine Claim~\ref{claim:condition_for_which_W_2_non_zero} and the observation~\eqref{eq:K_1_factor_into_connected_components} and see that any configuration $(\cB,\boldsymbol{\Ga},\Delta)$ from the sum~\eqref{eq:YMH_partition_function_forth_reduction} can be uniquely decomposed into polymers on
    \[
    \big\{ \mathbf{P} =  ({\sf P}_1,\ldots,{\sf P}_n) \, : \, {\sf P}_i \subset \La^\prime \  \text{are connected and no two elements touch} \big\} \, .
    \]
     With that, Definition~\ref{def:polymer} together with Fubini's theorem yields that
    \begin{equation*}
        Z_{\Lambda}^{\sf YMH,s}  = Z_{\Lambda,f}^{\sf G} \cdot \Xi_{\Lambda}(\mathbf{0}) \cdot\sum_{\mathbf{P}} \prod_{{\sf P}\in \mathbf{P}} {\sf w}({\sf P}) \, .
    \end{equation*}
    This is exactly the equality~\eqref{eq:polymer_expansion_for_YMH_partition_function_before_log}, and the proof of Proposition~\ref{proposition:expansion_for_YMH_partition_function_before_log} is complete. 
    \qed
    
    \section{Convergence of the cluster expansion}
    \label{sec:convergence_of_cluster_expansion}
    
    To prove our main result Theorem~\ref{thm:exponential_clustering_in_lattice_YMH}, we wish to take the logarithm on both sides of~\eqref{eq:polymer_expansion_for_YMH_partition_function_before_log} and then differentiate with respect to the `sources' to obtain the desired correlation; see the formula~\eqref{eq:diffrentiating_the_log_partition_function}. The procedure of taking the logarithm of a polymer expansion is fairly standard, and the resulting expansion is commonly known as the \emph{cluster expansion}. In what follows, we will mostly be following the terminology from the book~\cite[Chapter~5]{Friedli-Velenik_book}, although there are other excellent sources to learn about these expansions, see for instance~\cite{Cammarota, Kotecky-Preiss} or the book chapters~\cite[Chapter~3]{Seiler-book}  and~\cite[Chapter~V]{Simon-book}. 

    \subsection{The cluster expansion}
    Recall $\delta({\sf P}_1,{\sf P}_2) = 0$ if ${\sf P}_1,{\sf P}_2\subset \La^\prime$ touch while $\delta({\sf P}_1,{\sf P}_2)= 1$ otherwise, see~\eqref{eq:def_of_delta_interaction}. Following~\cite{Friedli-Velenik_book}, we say that two polymers ${\sf P}_1,{\sf P}_2\subset \La^\prime$ are \emph{compatible} if $\delta({\sf P}_1,{\sf P}_2) = 1$, and \emph{incompatible} if $\delta({\sf P}_1,{\sf P}_2) = 0$. We say that a collection $({\sf P}_1,\ldots,{\sf P}_n)$ is \emph{decomposable} if it is possible to write it as a disjoint union of two non-empty sets, in such a way that each ${\sf P}_i$ from one set is compatible with each ${\sf P}_j$ from the other set. An unordered, non-decomposable collection $\mathbf{X} = ({\sf P}_1,\ldots,{\sf P}_n)$ is called a \emph{cluster}. Given such cluster $\mathbf{X} = ({\sf P}_1,\ldots,{\sf P}_n)$ we set
    \begin{align}\label{eq:def_of_Psi}
        \Psi(\mathbf{X}) &=  \Big(\prod_{\substack{{\sf P}\subset \La^\prime \\ {\sf P} \text{ connected}}} \frac{1}{n_{\mathbf{X}}({\sf P})!} \Big) \cdot \Big( \prod_{{\sf P}\in \mathbf{X}} {\sf w}({\sf P}) \Big) \notag \\ 
        &\qquad \qquad \cdot \bigg(\sum_{G\in \{\text{connected graphs on } \mathbf{X}\}} \prod_{({\sf P}_i,{\sf P}_j)\in E(G)} \big(\delta({\sf P}_i,{\sf P}_j) - 1\big)\bigg) 
    \end{align}
    where $n_{\mathbf{X}}({\sf P})$ is the number of times the polymer ${\sf P}$ appears in the cluster $\mathbf{X}$. The combinatorial factor that appears in the bottom line of~\eqref{eq:def_of_Psi} is known as the \emph{Ursell function} of $\mathbf{X}$. For $X\subset \La^\prime$ we set 
    \begin{equation} \label{eq:def_of_psi_overline}
        \overline{\Psi}(X) = \sum_{\substack{\mathbf{X}= ({\sf P}_1,\ldots,{\sf P}_n) \\ \cup_i {\sf P}_i = X}} \Psi(\mathbf{X}) \, .
    \end{equation}
    Formally, the expansion from Proposition~\ref{proposition:expansion_for_YMH_partition_function_before_log} combined with~\cite[Proposition~5.3]{Friedli-Velenik_book} implies that
    \begin{align}
    \label{eq:cluster_expansion_for_log_YMH_partition_function} \nonumber
        \log Z_{\Lambda}^{{\sf YMH,s}} & = \log \big(Z_{\Lambda,f}^{{\sf G}}\cdot \Xi_{\Lambda} (\mathbf{0})\big) + \sum_{\mathbf{X}} \Psi(\mathbf{X}) \\ & = \log \big(Z_{\Lambda,f}^{{\sf G}}\cdot \Xi_{\Lambda} (\mathbf{0})\big) + \sum_{X\subset \La^\prime} \overline \Psi(X) \, .
    \end{align}
    The above equality is the desired \emph{cluster expansion} for $\log Z_{\Lambda}^{{\sf YMH,s}}$. We note that even though $\La$ (and hence $\La^\prime$) is finite, the sum on the right-hand side of~\eqref{eq:def_of_psi_overline} is in fact infinite, since the same polymer can appear arbitrary many times in a given cluster $\mathbf{X}$. Therefore, in what follows we will need to justify absolute convergence of the above sum, which will be our main technical challenge. For now, we state this as a claim for latter references.
    \begin{claim}
        \label{claim:cluster_expansion_converge}
        The series in~\eqref{eq:cluster_expansion_for_log_YMH_partition_function} converges absolutely.
    \end{claim}
    \subsection{A condition ensuring convergence}
    To check the convergence of~\eqref{eq:def_of_psi_overline} (and, in turn, the convergence of~\eqref{eq:cluster_expansion_for_log_YMH_partition_function}) it will be convenient to introduce a suitable norm, as done in~\cite[Section~9]{Balaban-Imbrie-Jaffe-Brydges} and~\cite[Chapter~V]{Simon-book}.
    \begin{definition}
        \label{def:r_norm}
        Let $f:\{{\sf P} : {\sf P}\subset L\bZ^d \} \to [0,\infty)$. For $r\ge 0$ we define its $r$-norm as 
        \[
        \|f \|_r = \sup_{x\in L\bZ^d} \sum_{\substack{{\sf P} \ \text{connected} \\ {\sf P} \  \text{touch } x }} f({\sf P}) \,  e^{r|{\sf P}|} \, .
        \]
        Similarly, given a function $f:\{ \Ga: \Ga\subset P^\prime(\bZ^d)\} \to [0,\infty)$ we define its $r$-norm
        \[
        \|f\|_r = \sup_{x\in L\bZ^d} \sum_{\substack{\Ga\ \text{connected} \\ \Ga \  \text{touch } x }} f(\Ga) \,  e^{r|X(\Ga)|} \, .
        \]
    \end{definition}
    \noindent 
    The $r$-norm is particularly convenient when checking convergence of cluster expansions. Proving the next theorem  which, in particular, ensures the convergence of the cluster expansion, will be our main objective in the sections ahead. 
    \begin{theorem}
        \label{theorem:condition_ensuring_convergence}
        Let ${\sf w}$ be given by~\eqref{eq:def_of_w}. Then for all $r\ge 0$ and for all $L\ge L_0(r,m,d)$ large enough, we have
        \[
        \lim_{\beta\to\infty} \| {\sf w} \|_r = 0 \, ,
        \]
        where $\beta$ is the inverse coupling constant for the Yang--Mills--Higgs measure~\eqref{eq:intro-def_of_YMH_measure}. 
    \end{theorem}
    The proof of Theorem~\ref{theorem:condition_ensuring_convergence} is given in Section~\ref{sec:breakdown_of_proof_condition_ensuring_convergence} below. Assuming Theorem~\ref{theorem:condition_ensuring_convergence} for the moment, it is not hard to check a suitable condition that ensures convergence of the cluster expansion, given that $\beta$ is sufficiently large.
    \begin{proof}[Proof of Claim~\ref{claim:cluster_expansion_converge}]
        Following~\cite[Theorem~5.4]{Friedli-Velenik_book}, the absolute convergence of~\eqref{eq:cluster_expansion_for_log_YMH_partition_function} would follow once we show that for each ${\sf P}\subset \La^\prime$ we have
        \begin{equation}
            \label{eq:condition_ensuring_convergence}
            \sum_{\substack{\widetilde {\sf P}\subset \La^\prime \ \text{connected} \\ {\sf P},\widetilde{{\sf P}} \ \text{touch}}} |{\sf w}(\widetilde {\sf P})| e^{|\widetilde{{\sf P}}|} \le |{\sf P}| \, .
        \end{equation}
        In view of Definition~\ref{def:r_norm}, we have
        \[
        \sum_{\substack{\widetilde {\sf P}\subset \La^\prime \ \text{connected} \\ {\sf P},\widetilde{{\sf P}} \ \text{touch}}} {\sf w}(\widetilde {\sf P}) \, e^{|\widetilde{{\sf P}}|} \le |{\sf P}| \cdot \sup_{x\in {\sf P}}  \sum_{\substack{{\sf P}\ \text{connected} \\ {\sf P} \  \text{touch } x }} {\sf w}({\sf P}) \,  e^{|{\sf P}|}  \le |{\sf P}| \cdot  \| {\sf w} \|_1\, .
        \]
        By Theorem~\ref{theorem:condition_ensuring_convergence}, for all $\beta$ large enough we have $\| {\sf w} \|_1\le 1$. For those values of $\beta$, condition~\eqref{eq:condition_ensuring_convergence} holds and the absolute convergence of the cluster expansion~\eqref{eq:cluster_expansion_for_log_YMH_partition_function} follows, as desired. 
    \end{proof}
    \subsection{Proof of the main result}
    In what follows, we would like to translate bounds on $\|{\sf w} \|_r$ (available to us through Theorem~\ref{theorem:condition_ensuring_convergence}) into bounds on the cluster expansion. This is made possible by the next claim. 
     \begin{claim}[{\cite[Appendix~A.2]{Balaban-Imbrie-Jaffe-Brydges}}]
        \label{claim:bound_of_norm_of_psi_by_norm_of_w}
        For $r\ge 0$, ${\sf w}$ given by~\eqref{eq:def_of_w} and $\overline \Psi$ given by~\eqref{eq:def_of_psi_overline}, we have
        \[
        \| \overline{\Psi} \|_r \le \sum_{n\ge 1}  \big(2 \| {\sf w} \|_{r+2}\big)^n \, .
        \]
    \end{claim}
    \noindent 
    We remark that similar bounds between $\overline\Psi$ and ${\sf w}$ also appear in~\cite{Cammarota, Kotecky-Preiss}, among other places. Assuming Theorem~\ref{theorem:condition_ensuring_convergence}, we conclude this section with the proof of our main result. 
    \begin{proof}[Proof of Theorem~\ref{thm:exponential_clustering_in_lattice_YMH}]
        By~\eqref{eq:cluster_expansion_for_log_YMH_partition_function} we have 
        \[
         \log Z_{\Lambda}^{{\sf YMH,s}} = \log \big(Z_{\Lambda,f}^{{\sf G}}\cdot \Xi_{\Lambda} (\mathbf{0})\big) + \sum_{X\subset \La^\prime} \overline \Psi(X) \, ,
        \]
        where the sum on the right-hand side is absolutely convergent for all $\beta$ large enough, in view of Claim~\ref{claim:cluster_expansion_converge}. Hence, \cite[Theorem~5.8]{Friedli-Velenik_book} yields that $$(t_x,t_y) \mapsto \log Z_{\Lambda}^{{\sf YMH,s}}$$ is analytic, where $t_x,t_y$ are the complex parameters~\eqref{eq:def_of_V_s} which correspond to the ``sources". Utilizing~\eqref{eq:diffrentiating_the_log_partition_function} and Cauchy's estimates~\cite[Theorem~2.2.7]{Hormander-SeveralComplexVariables} in the polydisk of radius $\beta^{-10}$, we see that
        \begin{align*}
                 \bigg|\bE_{\Lambda}^{\sf YMH}\big[\theta_x \theta_y\big] - \bE_{\Lambda}^{\sf YMH}\big[\theta_x \big]\cdot & \bE_{\Lambda}^{\sf YMH}\big[ \theta_y\big] \bigg| \\ &= \bigg| \frac{\partial}{\partial t_x} \frac{\partial}{\partial t_x} \bigg|_{t_x=t_y=0} \log Z_{\Lambda}^{\sf YMH,s} \bigg| \\ & = \bigg|\sum_{\substack{X\subset \La^\prime \\ \{x,y\} \subset \cE(X) }} \frac{\partial}{\partial t_x} \frac{\partial}{\partial t_x} \bigg|_{t_x=t_y=0} \overline{\Psi} (X) \bigg| \le C_{\beta} \sum_{\substack{X\subset \La^\prime \\ \{x,y\} \subset \cE(X) }} \sup_{|t_x|,|t_y| \le \beta^{-10}}\overline{\Psi} (X) \, .
        \end{align*}
        Since all $X$ in the above sum must be connected, the constraint $\{x,y\} \subset \cE (X)$ implies that $$|X|\ge c L^{-1} \cdot \text{dist}(x,y) \, ,$$  and Claim~\ref{claim:bound_of_norm_of_psi_by_norm_of_w} gives the bound
        \begin{align*}
            &\bigg|\bE_{\Lambda}^{\sf YMH}\big[\theta_x \theta_y\big] - \bE_{\Lambda}^{\sf YMH}\big[\theta_x \big]\cdot \bE_{\Lambda}^{\sf YMH}\big[ \theta_y\big] \bigg| \\
            &\le C_\beta \sum_{\substack{X\subset \La^\prime \\ \{x,y\} \subset \cE(X) }} \overline{\Psi} (X) e^{|X|} e^{-|X|} \le C_\beta \, e^{-cL^{-1}\text{dist}(x,y)} \sum_{X\subset \La^\prime} \overline{\Psi}(X) e^{|X|}  \\ &\le C_\beta \, e^{-cL^{-1}\text{dist}(x,y)} \| \overline \Psi \|_1 \stackrel{\text{Claim}~\ref{claim:bound_of_norm_of_psi_by_norm_of_w}}{\le} C_\beta \,   e^{-cL^{-1}\text{dist}(x,y)} \sum_{n\ge 1} (2 \| {\sf w} \|_3)^n \, . 
        \end{align*}
        By Theorem~\ref{theorem:condition_ensuring_convergence}, for all $\beta$ large enough we have $\|{\sf w}\|_3 < 1/3$, which yields
        \[
        \bigg|\bE_{\Lambda}^{\sf YMH}\big[\theta_x \theta_y\big] - \bE_{\Lambda}^{\sf YMH}\big[\theta_x \big]\cdot \bE_{\Lambda}^{\sf YMH}\big[ \theta_y\big] \bigg| \le C_\beta  \, e^{-c_L \, \text{dist}(x,y)} \, , 
        \]
        and the proof is complete. 
    \end{proof}
    
    \section{Breakdown of the proof of Theorem~\ref{theorem:condition_ensuring_convergence}}
    \label{sec:breakdown_of_proof_condition_ensuring_convergence}
    Theorem~\ref{theorem:condition_ensuring_convergence} asserts that the polymer weight ${\sf w}$ given by~\eqref{eq:def_of_w} has small $r$-norm as $\beta\to \infty$. The goal of this section is to reduce this statement into smaller, slightly more technical claims about the weights $\rho_{\mathcal{B}}, W_2,$ and $K_1$ that make up ${\sf w}$. 
    
    \subsection{Control on the different weights}
    Recall from Definition~\ref{def:good_and_bad_blocks} that $\mathcal{G}\subset \La^\prime$ is the collection of all good $L$-blocks, and that $\mathcal{B} =  \La^\prime \setminus \mathcal{G}$ are the bad $L$-blocks. We further recall that $\cB_1$ is the $r_\beta$-neighborhood of $\mathcal{B}$ is the coarse lattice $L\Z^d$, given by~\eqref{eq:def_of_B_1}. The next lemma shows that, typically, the set $\mathcal{B}_1$ will not have large connected components.  
    \begin{lemma}
        \label{lemma:large_field_values_are_exponentially_rare}
        For $X\subset \La^\prime$ connected we set
        \begin{equation}
            \label{eq:def_of_tilde_rho}
            \widetilde{\rho}(X) = \sum_{\substack{\mathcal{B} \subset  \La^\prime \\ \mathcal{B}_1 = X}} \bigg(\int_{[-\pi,\pi]^{\mathcal{\cE(\cB)}}} \rho_{\mathcal{B}}(\boldsymbol{\eta}) \, {\rm d} \boldsymbol{\eta}_{\mathcal{B}}\bigg)
        \end{equation}
        where $\rho_\mathcal{B}$ is given by~\eqref{eq:def_of_rho_B}. We further set $ \widetilde{\rho}(X) = 0$ if $X$ is not connected. Then for all $L\ge 1$ and for all $r\ge 0$ we have
        \[
        \lim_{\beta\to\infty} \| \widetilde{\rho} \|_r = 0 \, .
        \]
    \end{lemma}
    The proof of Lemma~\ref{lemma:large_field_values_are_exponentially_rare} is provided in Section~\ref{sec:large_field_values} below. Next, we state a lemma that deals with the Glimm-Jaffe-Spencer weights $W_2$ from Section~\ref{subsection:expansion_of_gaussian_partition}.
    \begin{lemma}
        \label{lemma:GJS_weights_are_small}
        For all $r\ge 0$ there exists $L\ge L_0(r,m,d)$ large enough and $C_L\ge 1$ such that 
        \[
        \| W_2(\cdot,\mathcal{G}) \|_r \le C_{L} \, ,
        \] 
        for all $\beta\ge 1$ and $\mathcal{G}\subset \La^\prime$. Here $W_2$ is given by~\eqref{eq:def_of_W_2}.
    \end{lemma}
    \noindent
    The proof of Lemma~\ref{lemma:GJS_weights_are_small} is given in Section~\ref{sec:the_GJS_expansion_for_gaussian_partition_functions}, where we also recall some standard Gaussian integration-by-parts formulas that will be useful in what follows. Finally, we need to deal with the weights $K_1$ given by~\eqref{eq:def_of_K_1}, which correspond to the non-Gaussian correction term. 
    For $\mathcal{B}_1$ given by~\eqref{eq:def_of_B_1} and $\mathcal{G}_1 = \La^\prime \setminus \mathcal{B}_1$ we set
    \begin{equation} \label{eq:def_of_C_B_1}
        \mathcal{C}(\mathcal{B}_1 ) = \big\{\Delta\subset P^\prime(\mathcal{G}_1) \, : \, \text{all connected components of $\Delta$ touch $\mathcal{B}_1$} \big\} \, .
    \end{equation}     
    The next lemma will be proved in Section~\ref{sec:dealing_with_non_gaussian_correction}.
    \begin{lemma}
    \label{lemma:bound_for_the_non_gaussian_weight}
       For all $r\ge 0$ and $L$ large enough there exists $C_L>0$ so that 
       \[
        \lim_{\beta\to\infty} \bigg( e^{-C_L|\mathcal{B}_1|}\sum_{\Delta\in \mathcal{C}(\mathcal{B}_1 ) } |K_1(\Delta;\mathcal{G})| \,  e^{r|X(\Delta)|} \bigg) = 0 \, . 
       \]
       Furthermore, when $\mathcal{B} = \mathcal{B}_1 = \emptyset$ (and consequently $\mathcal{G} = \mathcal{G}_1 = \La^\prime$), we have 
       \[
       \lim_{\beta\to \infty} \|K_1(\cdot,\La^\prime) \|_{r}  = 0 \, .
       \]
    \end{lemma}
    \noindent
    We also need to deal separately with the case $\Delta = \emptyset$, for which the next simple claim (proved in Section~\ref{sec:large_field_values} below) would be useful.
    \begin{claim}
    \label{claim:bound_on_K_1_emptyset}
        For all $L$ large enough there exists $C_L\ge 1$ so that
        \[
        K_1(\emptyset,\mathcal{G}) \le e^{C_L |\mathcal{B}_1|} \, ,
        \]
        for all $\mathcal{G}\subset  \La^\prime$. 
    \end{claim}

    \subsection{Combining the ingredients}
    We are almost ready to give the proof of Theorem~\ref{theorem:condition_ensuring_convergence}. Before doing so, we state two simple combinatorial facts that will be helpful in the proof, and also later on. 
    \begin{claim}[{\cite[Proposition~V.7.14]{Simon-book}}]
        \label{claim:bound_on_the_number_of_connected_sets}
        There exists $C>0$ so that for all $k\ge 1$ we have
        \[
        \# \big\{ X\subset \bZ^d \, : \, 0\in X, \   |X| = k, \ \text{and $X$ is connected} \big\} \le \exp(C k) \, .
        \]
    \end{claim}
    \begin{claim}
    \label{claim:bound_on_function_on_collection_of_subsets}
        Let $f$ be a non-negative function on subsets of $L\bZ^d$ and set
        \[
        F(\mathbf{X}) = \prod_{j=1}^n f({\sf P}_j) \, ,
        \]
        where $\mathbf{X} = ({\sf P}_1,\ldots,{\sf P}_n)$ is a collection of such  
        subsets. For all $A\ge 1$ and for all $\widetilde{{\sf P}} \subset L\bZ^d$, we have
        \[
        \sum_{n\ge 1} \frac{1}{n!} \bigg(\sum_{\substack{\mathbf{X} = ({\sf P}_1,\ldots,{\sf P}_n) \\ {\sf P}_j \  \text{and $\widetilde{{\sf P}}$ touch}}} F(\mathbf{X}) \bigg) \le \frac{1}{A} \exp\Big(A \, |\widetilde{{\sf P}}| \cdot \|f\|_0\Big) \, .
        \]
        A similar bound also holds for functions defined on the faces of $L\bZ^d$. 
    \end{claim}
    \begin{proof}
    We have
    \begin{align*}
        A \sum_{n\ge 1} \frac{1}{n!} & \bigg(\sum_{\substack{\mathbf{X} = ({\sf P}_1,\ldots,{\sf P}_n) \\ {\sf P}_j \  \text{and $\widetilde{{\sf P}}$ touch }}} F(\mathbf{X}) \bigg) \\ & = A \sum_{n\ge 1} \frac{1}{n!} \bigg(\sum_{{\sf P} \  \text{and $\widetilde{{\sf P}}$ touch}} f({\sf P}) \bigg)^n \\ &\le \sum_{n\ge 1} \frac{1}{n!} \bigg( A \sum_{{\sf P} \  \text{and $\widetilde{{\sf P}}$ touch}}   f({\sf P}) \bigg)^n \le \exp\Big(A \sum_{{\sf P} \  \text{and $\widetilde{{\sf P}}$ touch}}   f({\sf P}) \Big) \le \exp\Big(A \, |\widetilde{{\sf P}}| \cdot  \| f \|_0 \Big) \, .
    \end{align*}
    \end{proof}
    \noindent
    We conclude the section with the proof of Theorem~\ref{theorem:condition_ensuring_convergence}, assuming all of the above. 
    \begin{proof}[Proof of Theorem~\ref{theorem:condition_ensuring_convergence}]
        Recall that for ${\sf P}\subset \La^\prime$ connected we have 
        \begin{equation*}
        {\sf w}({\sf P}) = \sum_{(\mathcal{B},\mathbf{\Ga},\Delta) \ \text{is a polymer on } {\sf P}} \frac{1}{|\mathbf{\Ga}|!} \int_{[-\pi,\pi]^{\cE(\mathcal{B})}} \rho_{\mathcal{B}}(\boldsymbol{\eta}) \cdot   \prod_{\Ga\in \mathbf{\Ga}} W_2(\Ga;\mathcal{G}) \cdot K_1(\Delta;\mathcal{G}) \,  {\rm d}\boldsymbol{\eta}_{\mathcal{B}} \, , 
    \end{equation*}
    where the notion of a polymer is spelled out in Definition~\ref{def:polymer}. In particular, for any $(\mathcal{B},\mathbf{\Ga},\Delta)$ occurring in the above sum we have
    \begin{equation}
    \label{eq:bound_on_the_size_of_polymer_in_term_of_its_components}
        |{\sf P}| \le |\cB| + \sum_{\Ga\in \boldsymbol{\Ga}} |X(\Ga)| + |X(\Delta)| \, . 
    \end{equation}
    Let $R = r + C_d$ for some sufficiently large $C_d>0$ that we choose later. The bound~\eqref{eq:bound_on_the_size_of_polymer_in_term_of_its_components} implies that
    \begin{multline} \label{eq:bound_of_weight_w_after_introducing_R}
        |{\sf w}({\sf P})| e^{R|{\sf P}|}  \le  \sum_{(\mathcal{B},\mathbf{\Ga},\Delta) \ \text{is a polymer on } {\sf P}} \frac{1}{|\mathbf{\Ga}|!} \int_{[-\pi,\pi]^{\cE(\mathcal{B})}} \rho_{\mathcal{B}}(\boldsymbol{\eta}) \, e^{R|\cB|}  \\ \times   \Big(\prod_{\Ga\in \mathbf{\Ga}} W_2(\Ga;\mathcal{G}) \, e^{R|X(\Ga)|}\Big) \cdot K_1(\Delta;\mathcal{G}) \,  e^{R|X(\Delta)|} \,  {\rm d}\boldsymbol{\eta}_{\mathcal{B}} \, .
    \end{multline}
    To further bound the right-hand side of~\eqref{eq:bound_of_weight_w_after_introducing_R}, we split the sum according to whether $\mathcal{B} = \emptyset$ or $\mathcal{B}\not= \emptyset$. Denote by
    \begin{equation*}
        \mathcal{N}_1  =  \sum_{\substack{(\mathcal{B},\mathbf{\Ga},\Delta) \ \text{is a polymer on } {\sf P} \\ \mathcal{B}\not=\emptyset}} (\cdots)  \, ,  \qquad \text{and} \qquad \mathcal{N}_2 = \sum_{(\emptyset,\mathbf{\Ga},\Delta) \ \text{is a polymer on } {\sf P}} (\cdots) \, , 
    \end{equation*}
    where the summand in both sums is the same as in~\eqref{eq:bound_of_weight_w_after_introducing_R}. We start by estimating $\mathcal{N}_1$. Indeed, we first estimate the sum over $\boldsymbol{\Ga}$ while holding $\mathcal{B}$ and $\Delta$ fixed. By Claim~\ref{claim:condition_for_which_W_2_non_zero}, $W_2(\Ga,\mathcal{G}) = 0$ unless $\Ga$ is connected and touches $\mathcal{B}_1$. Hence, Claim~\ref{claim:bound_on_function_on_collection_of_subsets} (applied with $A=1$ and $\widetilde{{\sf P}} = \cB_1$) shows that
    \begin{equation*}
        \sum_{\boldsymbol{\Ga}} \frac{1}{|\boldsymbol{\Ga}|!} \Big(\prod_{\Ga\in \mathbf{\Ga}} W_2(\Ga;\mathcal{G}) \, e^{R|X(\Ga)|}\Big)  \stackrel{\text{Claim}~\ref{claim:bound_on_function_on_collection_of_subsets}}{ \le} \exp\Big( \| W_2(\cdot;\mathcal{G}) \|_R \cdot |\mathcal{B}_1|\Big) \stackrel{\text{Lemma}~\ref{lemma:GJS_weights_are_small}}{\le } \exp\Big(C_L |\mathcal{B}_1| \Big)
    \end{equation*}
    for all $L$ sufficiently large and $\beta \ge 1$. We get that
    \begin{equation} \label{eq:bound_on_weight_w_with_R_first}
        \mathcal{N}_1 \le \sum_{\substack{(\mathcal{B},\Delta) \\ \mathcal{B}\not=\emptyset , \ \mathcal{B} \cup X(\Delta) \subset {\sf P} }} e^{C_L|\mathcal{B}_1|} \int_{[-\pi,\pi]^{\cE(\mathcal{B})}} \rho_{\mathcal{B}}(\boldsymbol{\eta})   \cdot K_1(\Delta;\mathcal{G}) \,  e^{R|X(\Delta)|} \,  {\rm d}\boldsymbol{\eta}_{\mathcal{B}} \, .
    \end{equation}
    Moving on, we fix $\mathcal{B}$ and the components of $\Delta$ which touch $\mathcal{B}_1$ and bound the sum over the components strictly inside $\mathcal{G}_1$. Since $K_1(\Delta;\mathcal{G})$ factors across these connected components (see~\eqref{eq:K_1_factor_into_connected_components}), and since $K_1(\Delta;\mathcal{G}) = K_1(\Delta;\La^\prime)$ for $\Delta$ strictly inside $\mathcal{G}_1$,  we can apply Claim~\ref{claim:bound_on_function_on_collection_of_subsets} once more (with $A=1$ and $\widetilde {\sf P} = {\sf P}$) and get that the corresponding sum is less than
    \[
    \exp\Big(\| K_1( \cdot,\La^\prime) \|_{R} \cdot |{\sf P}| \Big) \stackrel{\text{Lemma}~\ref{lemma:bound_for_the_non_gaussian_weight}}{\le} \exp(|{\sf P}|)
    \]
    for all $\beta$ large enough. To deal with the components of $\Delta$ that touch $\mathcal{B}_1$ (that is, $\Delta$ which belong to the set $\mathcal{C}(\mathcal{B}_1)$ given by~\eqref{eq:def_of_C_B_1}), we apply Lemma~\ref{lemma:bound_for_the_non_gaussian_weight} to get that the corresponding sum is  less than
    \[
       \sum_{\Delta\in \mathcal{C}(\mathcal{B}_1 ) } |K_1(\Delta;\mathcal{G})| \,  e^{R|X(\Delta)|} \le e^{C_L|\mathcal{B}_1|} \, , 
    \]  
    for all $\beta$ large. Finally, we need to consider the case where $\Delta = \emptyset$, for which Claim~\ref{claim:bound_on_K_1_emptyset} shows that
    \[
    |K_1(\emptyset;\mathcal{G})| \le e^{C_L |\mathcal{B}_1|} \, . 
    \]
    Plugging all of these observations into~\eqref{eq:bound_on_weight_w_with_R_first}, we arrive at the bound 
    \begin{equation} \label{eq:bound_on_weight_w_with_R_second}
        \mathcal{N}_1 \le  e^{|{\sf P}|} \sum_{\substack{\mathcal{B}\not=\emptyset \\  \cB  \subset {\sf P} }} e^{C_L|\mathcal{B}_1|} \int_{[-\pi,\pi]^{\cE(\mathcal{B})}} \rho_{\mathcal{B}}(\boldsymbol{\eta}) \,  {\rm d}\boldsymbol{\eta}_{\mathcal{B}} \, .
    \end{equation}
    We note that the sum over $\mathcal{B}$ on the right-hand side of~\eqref{eq:bound_on_weight_w_with_R_second} is nothing but
    \[
    \sum_{\substack{\mathcal{B}\not=\emptyset \\  \cB  \subset {\sf P} }} e^{C_L|\mathcal{B}_1|} \int_{[-\pi,\pi]^{\cE(\mathcal{B})}} \rho_{\mathcal{B}}(\boldsymbol{\eta}) \,  {\rm d}\boldsymbol{\eta}_{\mathcal{B}}  = \sum_{\emptyset \not= X\subset {\sf P} } \prod_{\substack{ x \subset X \\ \text{ $x$ is a maximal connected component of $X$ }}} \big( \widetilde{\rho}(x) \, e^{C_L|x|}\big) \, ,
    \]
    where $\widetilde{\rho}$ is given by~\eqref{eq:def_of_tilde_rho}. Applying Claim~\ref{claim:bound_on_function_on_collection_of_subsets} with $A = \| \widetilde{\rho} \|_{C_L}^{-1}$ and $\widetilde{\sf P} = {\sf P}$, we get that 
    \[
    \sum_{\substack{\mathcal{B}\not=\emptyset \\  \cB  \subset {\sf P} }} e^{C_L|\mathcal{B}_1|} \int_{[-\pi,\pi]^{\cE(\mathcal{B})}} \rho_{\mathcal{B}}(\boldsymbol{\eta}) \,  {\rm d}\boldsymbol{\eta}_{\mathcal{B}}  \le \| \widetilde \rho \|_{C_L} \cdot e^{|{\sf P}|} \, ,
    \]
    and by plugging into~\eqref{eq:bound_on_weight_w_with_R_second} we get
    \begin{equation}\label{eq:bound_on_N_1}
        \mathcal{N}_1 \le e^{2|{\sf P}|}  \, \| \widetilde \rho \|_{C_L} \, .
    \end{equation}
    To deal with the sum $\mathcal{N}_2$, we note that Claim~\ref{claim:condition_for_which_W_2_non_zero} together with the fact that $\cB = \emptyset$ implies that $W_2(\Ga;\La) = 0$ for all $\Ga$. Thus, the only way to form a polymer when is to have $\boldsymbol{\Ga} = \emptyset$ and $X(\Delta) = {\sf P}$. This means that 
    \[
    \mathcal{N}_2 =  \sum_{ \substack{\Delta \  \text{such that } \\ X(\Delta) = {\sf P}} }  |K_1(\Delta;\La^\prime)| \,  e^{R|X(\Delta)|} =  e^{R|{\sf P}|} \sum_{ \substack{\Delta \  \text{such that } \\ X(\Delta) = {\sf P}} }  |K_1(\Delta;\La^\prime)| \, .
    \]
    Combining the above with the bound~\eqref{eq:bound_on_N_1}, we get from~\eqref{eq:bound_of_weight_w_after_introducing_R} that 
    \[
    |{\sf w}({\sf P})| \, e^{R|{\sf P}|}  \le  \mathcal{N}_1 + \mathcal{N}_2 \le  e^{R|{\sf P}|} \sum_{ \substack{\Delta \  \text{such that } \\ X(\Delta) = {\sf P}} }  |K_1(\Delta;\La^\prime)| + e^{2|{\sf P}|} \cdot \| \widetilde \rho \|_{C_L} \, ,
    \]
    which implies
    \begin{equation}\label{eq:bound_on_w_after_everything}
        |{\sf w}({\sf P})| \le \sum_{ \substack{\Delta \  \text{such that } \\ X(\Delta) = {\sf P}} }  |K_1(\Delta;\La^\prime)| + e^{(2-R)|{\sf P}|} \cdot \| \widetilde \rho \|_{C_L} \, .
    \end{equation}
    To conclude, we observe that Claim~\ref{claim:bound_on_the_number_of_connected_sets} implies that for any $x\in L\bZ^d$
    \[
    \sum_{\substack{{\sf P} \ \text{connected} \\ {\sf P} \  \text{touch } x }} e^{(2-R)|{\sf P}|} \,  e^{r|{\sf P}|} = \sum_{\substack{{\sf P} \ \text{connected} \\ {\sf P} \  \text{touch } x }} e^{(2-C_d)|{\sf P}|} < \infty \, , 
    \]
    assuming that $C_d$ is sufficiently large. Furthermore,  
    \[
    \sum_{\substack{{\sf P} \ \text{connected} \\ {\sf P} \  \text{touch } x }}  e^{r|{\sf P}|} \Big( \sum_{ \substack{\Delta \  \text{such that } \\ X(\Delta) = {\sf P}} }  |K_1(\Delta;\La^\prime)| \Big) =  \sum_{\substack{\Delta \ \text{connected} \\ \Delta \  \text{touch } x }} e^{r|X(\Delta)|} |K_1(\Delta;\La^\prime)| \le \| K_1(\cdot , \La^\prime) \|_r \, .
    \]
    In view of~\eqref{eq:bound_on_w_after_everything}, the last two inequalities implies that 
    \[
    \| {\sf w} \|_r \le \| K_1(\cdot , \La^\prime) \|_r + C \, \| \widetilde \rho \|_{C_L} \, ,
    \]
    and the proof now follows by combining Lemma~\ref{lemma:large_field_values_are_exponentially_rare} and Lemma~\ref{lemma:bound_for_the_non_gaussian_weight}. 
    \end{proof}
    
    \section{Large field values}
    \label{sec:large_field_values}
  
    In this section we prove Lemma~\ref{lemma:large_field_values_are_exponentially_rare}. In turn, this would follow from a stability inequality for the Yang--Mills--Higgs action~\eqref{eq:intro_hamiltoniam_with_beta_and_mass}, i.e., showing it is lower bounded by a Gaussian action. 
    \subsection{Stability of the Yang--Mills--Higgs action}
    Recall that $\mathcal{H}_{\La,{\sf s}}(\boldsymbol{\theta}) = \mathcal{H}_{\La}(\boldsymbol{\theta}) + V_{\sf s}(\boldsymbol{\theta})$, where 
    \[
    \mathcal{H}_{\La}(\boldsymbol{\theta}) = \sum_{p\in P(\Lambda)} \big[1 - \cos\big({\sf d} \theta_p\big)\big] + m \sum_{e\in E(\Lambda)} \big[1- \cos\big(\theta_e\big)\big] \, ,
    \]
    and $V_{\sf s}(\boldsymbol{\theta})$, which is given by~\eqref{eq:def_of_V_s}, is the ``sources" term in the action. 
    \begin{claim}\label{claim:YMH_action_is_stable}
        For all $\boldsymbol{\theta} \in [-\pi,\pi]^{E(\La)}$ we have 
        \[
        \mathcal{H}_{\La}(\boldsymbol{\theta}) \ge \frac{1}{10}\sum_{p\in P(\Lambda)} \big({\sf d} \theta_p\big)^2 + \frac{m}{10} \sum_{e\in E(\Lambda)} \big(\theta_e\big)^2 \stackrel{\eqref{eq:def_of_gaussian_action}}{=} \frac{1}{5} \cdot S_{\La}(\boldsymbol{\theta}) \, . 
        \]
        Furthermore, recalling from~\eqref{eq:action_restricted_to_H} that $V_{\Lambda}^{\mathcal{B}}(\boldsymbol{\eta})$ and $S_{\Lambda}^{\mathcal{B}}(\boldsymbol{\eta})$ are the restrictions on $V_{\Lambda}$ and $S_{\Lambda}$ to $\cB$, respectively, we have that 
        \[
        V_{\Lambda}^{\mathcal{B}}(\boldsymbol{\eta}) + S_{\Lambda}^{\mathcal{B}}(\boldsymbol{\eta}) \ge \frac{1}{5} S_{\Lambda}^{\mathcal{B}}(\boldsymbol{\eta}) + V_{\sf s}(\boldsymbol{\eta}) \, ,
        \]
        for all $\boldsymbol{\eta} \in [-\pi,\pi]^{\cE(\cB)}$. 
    \end{claim}
    \begin{proof}
        Both statements in the claim follows immediately from the elementary inequality
        \[
        1- \cos(\theta) \ge \frac{\theta^2}{10}
        \]
        valid for all $\theta\in[-\pi,\pi]$. 
    \end{proof}
    \noindent
    The proof of Lemma~\ref{lemma:large_field_values_are_exponentially_rare} requires us to obtain a bound on the ``large field values" density, which by~\eqref{eq:def_of_rho_B} is given by
    \begin{equation*}
        \rho_{\mathcal{B}}(\boldsymbol{\eta}) = e^{-\beta V_{\Lambda}^{\mathcal{B}}(\boldsymbol{\eta})- \beta S_{\Lambda}^{\mathcal{B}}(\boldsymbol{\eta})} \cdot \zeta_{\mathcal{B}}(\boldsymbol{\eta}) \cdot \frac{Z_{\mathcal{G},\eta}^{\sf G}(\mathbf{0})}{Z_{\Lambda,f}^{\sf G}(\mathbf{0})} \cdot \frac{\Xi_{\mathcal{G},\eta}(\mathbf{0},+)}{\Xi_{\Lambda}(\mathbf{0})} \, .
    \end{equation*}
    The next simple claim will be used several times in what follows. 
    \begin{claim}
        \label{claim:two_sided_inequality_for_ratio_of_xi}
        Let $\Xi_{\mathcal{G},\eta}(\mathbf{0},+)$ and $\Xi_{\Lambda}(\mathbf{0})$ be given by~\eqref{eq:def_of_Xi_0_+} and~\eqref{eq:def_of_Xi_with_interpolation_s}, respectively. Then we have 
        \[
        e^{-C_L|\mathcal{B}_1|} \le \frac{\Xi_{\mathcal{G},\eta}(\mathbf{0},+)}{\Xi_{\Lambda}(\mathbf{0})} \le e^{C_L |\mathcal{B}_1|}
        \]
        for some $C_L >0$ and for large enough $\beta$. 
    \end{claim}
    \begin{proof}
        Looking at the ratio
        \begin{equation}
        \label{eq:ratio_of_Xi_proof_of_claim_two_sided_inequality}
        \frac{\Xi_{\mathcal{G},\eta}(\mathbf{0},+)}{\Xi_{\Lambda}(\mathbf{0})} \, ,    
        \end{equation}
        we see that both numerator and denominator in~\eqref{eq:ratio_of_Xi_proof_of_claim_two_sided_inequality} factor into a product over the different $L$-blocks in $\mathcal{G}$, and that any $L$-block which does not touch $\mathcal{B}_1$ cancels between the numerator and the denominator exactly. To bound from above and below the ratio between $L$-blocks in~\eqref{eq:ratio_of_Xi_proof_of_claim_two_sided_inequality} which do not cancel, we note the simple inequality 
        \begin{equation}\label{eq:proof_of_claim:two_sided_inequality_for_ratio_of_xi_role_of_chi}
            \mathbf{1}_{\{|\theta| \le T_\beta \}} \le \chi \Big(\frac{|\theta|}{T_\beta}\Big) e^{-\beta g(\theta)} \le 2 \, ,
        \end{equation}
        where $\chi$ is the smooth cut-off function from the partition of unity~\eqref{eq:partition_of_unity_into_large_and_small_fields}. Indeed, the upper bound in~\eqref{eq:proof_of_claim:two_sided_inequality_for_ratio_of_xi_role_of_chi} follows as on the support of $\chi$ we have
        \[
        -\beta g(\theta) = \beta \cdot \big( \cos(\theta) - 1 + \theta^2/2\big) \lesssim \beta T_\beta^{-4} \lesssim \frac{\log^{O(1)}(\beta)}{\beta} \, .
        \]
        On the other hand, the lower bound in~\eqref{eq:proof_of_claim:two_sided_inequality_for_ratio_of_xi_role_of_chi} follows from non-positivity of $g(\theta)$ on the support of $\chi$, and the fact that $\chi\big(\frac{\cdot}{T_\beta}\big)  \equiv 1$ for $|\theta| \le T_\beta$. 
        Letting $Q=Q(v)$ be a symbolic $L$-block in $v\in \La^\prime$, we denote by
        \begin{equation*}
        \chi_Q (\boldsymbol{\theta}) = \prod_{e\in E(Q)} \chi\Big(\frac{|\theta_e|}{T_\beta}\Big) \, .
        \end{equation*}
        The inequality~\eqref{eq:proof_of_claim:two_sided_inequality_for_ratio_of_xi_role_of_chi} immediately implies that 
        \begin{equation} \label{eq:two_sided_bound_for_corrector_in_given_L_block}
        \frac{1}{C_L} \le \int_{\bR^{E(Q)}} \chi_{Q}(\boldsymbol{\theta})  e^{-\beta V_{Q}(\boldsymbol{\theta})} \, {\rm d}\mu_{Q,f}^{\sf G}(\boldsymbol{\theta}) \le C_L    
        \end{equation}
        for all large enough $\beta$, where the underlying Gaussian field in given by Definition~\ref{def:lattice_proca_field} and $V_Q$ is given by~\eqref{eq:def_of_V}. From here, as the only non-trivial contributions to the ratio~\eqref{eq:ratio_of_Xi_proof_of_claim_two_sided_inequality} is given by blocks which touch $\mathcal{B}_1$, we see that 
        \[
        \Big(\frac{1}{C_L}\Big)^{|\mathcal{B}_1|} \le  \frac{\Xi_{\mathcal{G},\eta}(\mathbf{0},+)}{\Xi_{\Lambda}(\mathbf{0})} \le C_L^{|\mathcal{B}_1|} \, , 
        \]
        as desired.
    \end{proof}
    Combining Claim~\ref{claim:YMH_action_is_stable} and Claim~\ref{claim:two_sided_inequality_for_ratio_of_xi}, we derive the main estimate for the ``large field" density, given by the next lemma.
    \begin{lemma}
        \label{lemma:estimate_on_integral_over_large_fields}
        Let $\cB\subset \La^\prime$ be the set of bad $L$-blocks and let $\mathcal{B}_1$ be given by~\eqref{eq:def_of_B_1}. There exists constants $C_L,c_L>0$ so that 
        \[
        \int_{[-\pi,\pi]^{\cE(\mathcal{B})}} \rho_{\mathcal{B}}(\boldsymbol{\eta}) \, {\rm d} \boldsymbol{\eta}_{\mathcal{B}} \le \exp\Big( C_L |\mathcal{B}_1| \cdot (\log \beta) - c_L \beta \, T_\beta^2  \, |\mathcal{B}| \Big) \, ,
        \]
        where $T_\beta$ is given by~\eqref{eq:def_of_T_beta}.
    \end{lemma}
    \begin{proof}
        Claim~\ref{claim:YMH_action_is_stable} implies that  
        \begin{align} \label{eq:proof_of_lemma:estimate_on_integral_over_large_fields_energy_estimate}
          \nonumber 
          \int_{[-\pi,\pi]^{\cE(\mathcal{B})}}  e^{-\beta V_{\Lambda}^{\mathcal{B}}(\boldsymbol{\eta})- \beta S_{\Lambda}^{\mathcal{B}}(\boldsymbol{\eta})} & \zeta_{\mathcal{B}}(\boldsymbol{\eta}) \, {\rm d} \boldsymbol{\eta}_{\mathcal{B}}  \\ &\stackrel{\text{Claim}~\ref{claim:YMH_action_is_stable}}{\le} \int_{[-\pi,\pi]^{\cE(\mathcal{B})}} e^{-\frac{\beta}{5} S_{\La}^{\mathcal{B}}(\boldsymbol{\eta}) - \beta V_{\sf s}(\boldsymbol{\eta}) } \zeta_{\mathcal{B}}(\boldsymbol{\eta}) \, {\rm d} \boldsymbol{\eta}_{\mathcal{B}}  \le e^{-c_L \beta  \, T_\beta^2 \, |\mathcal{B}|} \, .
        \end{align}
        The second inequality in the display above follows from the fact that $\zeta_{\mathcal{B}}$, given by~\eqref{eq:def_of_zeta_B}, forces at least one edge per $L$-block in $\cB$ to have $|\theta_e| \ge T_\beta$. Furthermore, we have
        \begin{equation} \label{eq:proof_of_lemma:estimate_on_integral_over_large_fields_ratio_of_gaussian_partitions}
            \frac{Z_{\mathcal{G},\eta}^{\sf G}(\mathbf{0})}{Z_{\Lambda,f}^{\sf G}(\mathbf{0})} \le \beta^{C_L |\mathcal{B}_1|} \, .            
        \end{equation}
        Indeed, as in the proof of Claim~\ref{claim:two_sided_inequality_for_ratio_of_xi}, both numerator and denominator factor across $L$-blocks, and those which do not touch $\mathcal{B}_1$ cancel exactly. In the numerator, blocks which do not cancel can be bounded by $1$ trivially, while for any such $L$-block $Q$ in the denominator a simple Gaussian computation shows
        \[
        \int_{\bR^{E(Q)}} e^{-\beta S_{Q}(\boldsymbol{\theta})} \, \prod_{e\in E(Q)} {\rm d} \theta_e \gtrsim \beta^{-|E(Q)|/2}  = \beta^{-C_L} \, .
        \]
        Combining these observations yields the bound~\eqref{eq:proof_of_lemma:estimate_on_integral_over_large_fields_ratio_of_gaussian_partitions}. Finally, by combining~\eqref{eq:proof_of_lemma:estimate_on_integral_over_large_fields_energy_estimate},~\eqref{eq:proof_of_lemma:estimate_on_integral_over_large_fields_ratio_of_gaussian_partitions} and the upper bound in Claim~\ref{claim:two_sided_inequality_for_ratio_of_xi} we get that
        \begin{align*}
            \int_{[-\pi,\pi]^{\cE(\mathcal{B})}} \rho_{\mathcal{B}}(\boldsymbol{\eta}) \, {\rm d} \boldsymbol{\eta}_{\mathcal{B}} & = \int_{[-\pi,\pi]^{\cE(\mathcal{B})}}  e^{-\beta V_{\Lambda}^{\mathcal{B}}(\boldsymbol{\eta})- \beta S_{\Lambda}^{\mathcal{B}}(\boldsymbol{\eta})}\zeta_{\mathcal{B}}(\boldsymbol{\eta}) \cdot \frac{Z_{\Lambda,\eta}^{\sf G}(\mathbf{0})}{Z_{\Lambda,f}^{\sf G}(\mathbf{0})} \cdot \frac{\Xi_{\mathcal{G},\eta}(\mathbf{0},+)}{\Xi_{\Lambda}(\mathbf{0})}  \, {\rm d} \boldsymbol{\eta}_{\mathcal{B}}  \\ & \le \beta^{C_L|\mathcal{B}_1|} \exp\Big(- c_L \beta \, T_\beta^2  \, |\mathcal{B}| \Big) \\ &=  \exp\Big( C_L |\mathcal{B}_1| \cdot (\log \beta) - c_L \beta \, T_\beta^2  \, |\mathcal{B}| \Big) \, ,
        \end{align*}
        and the lemma follows. 
    \end{proof}
    \noindent
    Before we conclude this section with the proof of Lemma~\ref{lemma:large_field_values_are_exponentially_rare}, we show that Claim~\ref{claim:bound_on_K_1_emptyset} is in fact a simple corollary of Claim~\ref{claim:two_sided_inequality_for_ratio_of_xi}. 
    \begin{proof}[Proof of Claim~\ref{claim:bound_on_K_1_emptyset}]
        Note that
        \[
        K_1(\emptyset,\mathcal{G}) = \frac{\Xi_{\mathcal{G},\eta}(\mathbf{0})}{\Xi_{\mathcal{G},\eta}(\mathbf{0},+)} 
        \]
        where $\Xi_{\mathcal{G},\eta}(\mathbf{0})$ is given by~\eqref{eq:def_of_Xi_with_interpolation_s} and $\Xi_{\mathcal{G},\eta}(\mathbf{0},+)$ is given by~\eqref{eq:def_of_Xi_0_+}. By Claim~\ref{claim:two_sided_inequality_for_ratio_of_xi} we have 
        \begin{equation} \label{eq:proof_of_claim:bound_on_K_1_emptyset_after_application}
             K_1(\emptyset,\mathcal{G}) = \frac{\Xi_{\La}(\mathbf{0})}{\Xi_{\mathcal{G},\eta}(\mathbf{0},+)} \cdot \frac{\Xi_{\mathcal{G},\eta}(\mathbf{0})}{\Xi_{\La}(\mathbf{0})} \le e^{C_L|\mathcal{B}_1|} \cdot \frac{\Xi_{\mathcal{G},\eta}(\mathbf{0})}{\Xi_{\La}(\mathbf{0})} \, .
        \end{equation}
        It remains to bound the ratio appearing on the right-hand side of~\eqref{eq:proof_of_claim:bound_on_K_1_emptyset_after_application}. Indeed, such a bound follows from a similar argument as in the proof of Claim~\ref{claim:two_sided_inequality_for_ratio_of_xi}. Both numerator and denominator factor across $L$-blocks, and those which do not touch $\mathcal{B}_1$ cancel exactly. To deal with $L$-blocks in the numerator which did not get canceled, we simply apply the upper bound in the inequality~\eqref{eq:proof_of_claim:two_sided_inequality_for_ratio_of_xi_role_of_chi}. On the other hand, the desired lower bound on $L$-blocks in the denominator is trivial, as $\Xi_{\La}(\mathbf{0})$ factors across \emph{all} $L$-blocks in $\La^\prime$. Combining all of the above, we get that
        \[
        \frac{\Xi_{\mathcal{G},\eta}(\mathbf{0})}{\Xi_{\La}(\mathbf{0})}  \le e^{C_L |\mathcal{B}_1|} \, ,
        \]
        which, in view of~\eqref{eq:proof_of_claim:bound_on_K_1_emptyset_after_application}, concludes the proof.
    \end{proof}
    \subsection{Proof of Lemma~\ref{lemma:large_field_values_are_exponentially_rare}}
    Recalling that $T_\beta = \log^{d+2}(\beta)/\sqrt{\beta}$, we see from Lemma~\ref{lemma:estimate_on_integral_over_large_fields} that 
    \begin{align*}
        \widetilde{\rho}(X) &= \sum_{\substack{\mathcal{B} \subset \La^\prime \\ \mathcal{B}_1 = X}} \bigg(\int_{[-\pi,\pi]^{\cE(\mathcal{B})}} \rho_{\mathcal{B}}(\boldsymbol{\eta}) \, {\rm d} \boldsymbol{\eta}_{\mathcal{B}}\bigg) \\ &\le \sum_{\substack{\mathcal{B} \subset \La^\prime \\ \mathcal{B}_1 = X}} \exp\Big( C_L |\mathcal{B}_1| \cdot (\log \beta) - c_L \beta \, T_\beta^2  \, |\mathcal{B}| \Big) \\ &= e^{-\log(\beta)\, |X|} \sum_{\substack{\mathcal{B} \subset \La^\prime \\ \mathcal{B}_1 = X}} \exp\Big( (C_L +1 )|\mathcal{B}_1| \cdot (\log \beta) - c_L \log^{2d+4}(\beta)  \cdot |\mathcal{B}| \, \Big) \, .
    \end{align*}
    By the definition~\eqref{eq:def_of_B_1} of $\mathcal{B}_1$, we have the crude bound
    \[
    |\mathcal{B}_1| \lesssim r_\beta^d \cdot |\mathcal{B}| = \log^{2d}(\beta) \cdot |\mathcal{B}| \, , 
    \]
    and we get that
    \begin{equation*}
        \widetilde{\rho}(X) \le e^{-\log(\beta)|X|} \cdot  \sum_{\mathcal{B}\subset X} \eps_\beta^{|\mathcal{B}|},
    \end{equation*}
    where $\eps_\beta =  \exp\big(\log^{2d+2}( \beta ) - c_L \log^{2d+4}(\beta)  \big)$. Clearly, $\eps_\beta$ can be made arbitrarily small by taking $\beta$ to be sufficiently large. Since $\widetilde f(\mathcal{B}) = \eps_\beta^{|\mathcal{B}|}$ factors across connected components of $\mathcal{B}$, we can combine Claim~\ref{claim:bound_on_function_on_collection_of_subsets} with Claim~\ref{claim:bound_on_the_number_of_connected_sets} and get that 
    \[
    \sum_{\mathcal{B}\subset X} \eps_\beta^{|\mathcal{B}|} \stackrel{\text{Claim}~\ref{claim:bound_on_function_on_collection_of_subsets}}{\le}  \exp\big(|X| \cdot\| \widetilde{f}\|_0\big) \stackrel{\text{Claim~\ref{claim:bound_on_the_number_of_connected_sets}}}{\le}  \exp \big(|X|\big) \, ,
    \]
    for all $\beta$ large enough. Altogether, we got that
    \begin{equation*}
        \widetilde{\rho}(X) \le  e^{-(\log(\beta)-1) \, |X|} \, ,
    \end{equation*}
    from which, in view of Claim~\ref{claim:bound_on_the_number_of_connected_sets}, the desired result follows. 
    \qed 
    
    \section{The Glimm--Jaffe--Spencer expansion}
    \label{sec:the_GJS_expansion_for_gaussian_partition_functions}
 
    The goal of this section is to prove Lemma~\ref{lemma:GJS_weights_are_small}. Along the way, we will also develop some necessary tools to deal with the non-Gaussian correction terms present in our cluster expansion (and, ultimately, prove Lemma~\ref{lemma:bound_for_the_non_gaussian_weight} in Section~\ref{sec:dealing_with_non_gaussian_correction} below). Recall that for $\Ga\subset P^\prime(\mathcal{G}_1)$ we have 
    \[
    W_2(\Ga;\mathcal{G}) = W_1(\Ga;\mathcal{G}) - W_1(\Ga;\La) \, ,
    \]
    where 
    \[
    W_1(\Ga;\mathcal{G}) = \int_{[0,1]^{\Ga}} \partial^\Ga \log Z_{\mathcal{G},\eta}^{\sf G} (\mathbf{s}_\Ga) \, {\rm d} s_{\Ga}\, .
    \]
    Before we move on to bound these weights, and with that prove Lemma~\ref{lemma:GJS_weights_are_small}, we first collect some basic fact about differentiation of Gaussian measures. 
    
    \subsection{Some Gaussian preliminaries}
    Recall that $\mu_{\mathcal{G}, \eta, \mathbf{s}}^{\sf G}$ and $Z_{\mathcal{G},\eta ,\mathbf{s}}^{\sf G}$ are the probability measure and the partition functions, respectively, which correspond to the (interpolated) Gaussian action $S_{\La,\mathcal{G}}(\cdot ; \mathbf{s})$ given by~\eqref{eq:def_of_gaussian_action_with_s}. Throughout this section, the set of good $L$-blocks $\mathcal{G}$ and the boundary conditions $\eta$ remain fixed. Therefore, we will lighten  the notation and denote by $Z_{\mathbf{s}}^{\sf G}$, $S(\cdot ; \mathbf{s})$ and $\mu_{\mathbf{s}}^{\sf G}$ the corresponding partition function, action and probability measure. We denote by $\bE_{\mathbf{s}}^{\sf G}$ the corresponding expectation operator and by 
    \begin{equation}
        \label{eq:covariance_of_interpolated_gaussian}
        \Cov_{\mathbf{s}}(\theta_x,\theta_y) = \bE_{\mathbf{s}}^{\sf G}\big[\theta_x \theta_y \big] - \bE_{\mathbf{s}}^{\sf G}\big[\theta_x \big]\cdot \bE_{\mathbf{s}}^{\sf G}\big[\theta_y \big] 
    \end{equation}
    the covariance for $x,y\in \cE(\mathcal{G})$. 
    \begin{lemma}
        \label{lemma:proca_field_is_massive}
        There exists $c,C>0$ depending only on the dimension $d\ge 2$ so that
        \[
        \big| \Cov_{\mathbf{s}}(\theta_x,\theta_y) \big| \le C \beta^{-1} e^{-c m \, \text{\normalfont dist}(x,y) }
        \]
        uniformly in $\mathbf{s}\in [0,1]^{P^\prime(\mathcal{G}_1)}$, $\beta\ge 1$ and the volume on $\La= \La_n$. 
    \end{lemma}
    \begin{proof}
        By scaling, we may assume that $\beta = 1$ for the proof of this lemma. Let $R_{m,\mathbf{s}} = R$ denote the inverse covariance matrix for the Gaussian measure $\mu_{\mathbf{s}}^{\sf G}$, that is
        \[
        R^{-1}(x,y) = \Cov_{\mathbf{s}}(\theta_x,\theta_y)
        \]
        for $x,y\in \cE(\mathcal{G})$. Recalling that the Gaussian action is given by
        \begin{equation*}
        S_{\Lambda,\mathcal{G}}(\boldsymbol{\theta};\mathbf{s}) = \frac{1}{2} \sum_{\substack{p\in P(\Lambda) \\ E(p)\cap \cE(\mathcal{G}) \not= \emptyset}}  \sigma_p(\mathbf{s}) \big({\sf d}\theta_p\big)^2 + \frac{m}{2}\sum_{e\in \cE(\mathcal{G})} (\theta_e)^2 = \frac{1}{2} \langle  \boldsymbol{\theta} , R \, \boldsymbol{\theta} \rangle \, , 
    \end{equation*}
    with $\sigma_p(\mathbf{s})$ given by~\eqref{eq:def_of_sigma_p}, we see that for all $\psi\in \ell^2(\cE(\mathcal{G}))$ we have
     \begin{equation*}
         \langle \psi , R \psi \rangle  \ge m\| \psi \|^2 \, . 
     \end{equation*}
     Furthermore, as each plaquette induces interaction only among finitely many edges, a simple application of Cauchy-Schwarz shows that
     \begin{equation*}
         \langle \psi , R \psi \rangle \le m \| \psi \|^2 + C_d \| \psi \|^2 \, ,
     \end{equation*}
     for some $C_d>0$ depending only on the lattice dimension, and uniformly over $\mathbf{s}$. We conclude that the spectrum of $R$ must lie in the interval $[m,m+C_d]$. With that in mind, we set
     \[
     S = I - \frac{1}{m+C_d} R
     \]
     where $I$ is the identity matrix on $\ell^2(\cE(\mathcal{G}))$, and note that $S$ is a non-negative definite matrix whose maximal eigenvalue is at most $1 - m/(m+C_d)<1$. We get that
     \[
     (m+C_d) R^{-1} = (I-S)^{-1} = \sum_{k=0}^{\infty} S^k \, ,
     \]
     where the series converges absolutely. Since the entries of the matrix $S$ are non-zero only in a band of finite width around the diagonal, we conclude that for some $c>0$
     \begin{align*}
         \big| \Cov_{\mathbf{s}}(\theta_x,\theta_y) \big| &\le \frac{1}{m+C_d} \sum_{k=0}^\infty |S(x,y)|^k \\ &= \frac{1}{m+C_d} \sum_{k= \lfloor c \, \text{dist}(x,y) \rfloor }^\infty |S(x,y)|^k \\ &\le \frac{1}{m+C_d} \sum_{k= \lfloor c \, \text{dist}(x,y) \rfloor }^\infty \| S\|^k  \\ & \le \frac{1}{m+C_d} \sum_{k= \lfloor c \, \text{dist}(x,y) \rfloor }^\infty  \Big(1-\frac{m}{m+C_d}\Big)^k \, ,
     \end{align*}
     from which the desired bound follows.  
    \end{proof}
    In what follows, we will also need to estimate derivatives of the covariance operator~\eqref{eq:covariance_of_interpolated_gaussian}. Recall that for $\Ga\subset P^\prime(\mathcal{G}_1)$ we defined 
    \[
    \partial^\Ga = \prod_{\gamma\in \Ga} \frac{\partial}{\partial s_\gamma} \, .
    \]
    For $x,y\in \cE(\cG)$ and $\Ga\subset P^\prime(\mathcal{G}_1)$, we denote by $d(x,y;\Ga)$ as the length of the shortest path (in $\bZ^d$) between $x$ and $y$ which visits at least one edge in different faces of $\Ga$. Similarly, we define $d(x,\mathcal{B}_1;\Ga)$ as the shortest path between $x$ and $\mathcal{B}_1$, which visits at least one edge in each different face of $\Ga$. 
    \begin{lemma}
        \label{lemma:estimate_on_derivative_of_covariance}
        There exists constants $c_1,c_2,c_3>0$ depending only on the lattice dimension $d\ge 2$, so that for any $\Ga\subset P^\prime(\mathcal{G}_1)$ we have 
        \[
        \sup_{\mathbf{s}} \big| \partial^\Ga \bE_{\mathbf{s}}^{\sf G}\big[\theta_x \big]\big| \le c_1 \beta^{-1} (c_2 L^d)^{|\Ga|} \,  e^{-c_3 m \, d(x,\mathcal{B}_1;\Ga)} \, , 
        \]
        and
        \[
        \sup_{\mathbf{s}} \big| \partial^\Ga \Cov_{\mathbf{s}}(\theta_x,\theta_y) \big| \le c_1 \beta^{-1} (c_2L^d)^{|\Ga|} \, e^{-c_3 m \, d(x,y ;\Ga)} \, .
        \]
    \end{lemma}
    \begin{proof}
        We denote by ${\sf d}^\ast {\sf d}(\mathbf{s})$ the linear operator on $\ell^2(\mathcal{G})$ defined via the relation
        \begin{equation} \label{eq:interaction_in_lattice_exterior_calculus_language}
            \langle \boldsymbol{\theta}  , {\sf d}^\ast {\sf d}(\mathbf{s}) \, \boldsymbol{\theta} \rangle = \sum_{\substack{p\in P(\Lambda) \\ E(p)\cap \cE(\mathcal{G}) \not= \emptyset}}  \sigma_p(\mathbf{s}) \big({\sf d}\theta_p\big)^2 \, . 
        \end{equation}
        The reason for this notation originates in the theory of exterior calculus on $\bZ^d$; let us briefly explain this connection. Indeed, following~\cite[Section~10.1]{Cao-Sheffield-FGF-overview} (see also~\cite[Section~2.3]{Frohlich-Spencer} or~\cite[Section~2.2]{Garban-Sepulveda-IMRN2023}) we denote the lattice exterior derivative by ${\sf d}$, see~\cite[Definition~10.7]{Cao-Sheffield-FGF-overview}. In our case, we can think of ${\sf d}$ as a mapping of functions defined on the edges of $\bZ^d$ (i.e., 1-forms) to functions defined on plaquettes of $\bZ^d$ (2-forms). Letting ${\sf d}^\ast$ be the lattice co-differential (given by~\cite[Definition~10.10]{Cao-Sheffield-FGF-overview}), which is simply the adjoint to ${\sf d}$, one can readily check that~\eqref{eq:interaction_in_lattice_exterior_calculus_language} holds for ${\sf d^\ast d} = {\sf d^\ast d}(\mathbf{1})$. Then ${\sf d^\ast d}(\mathbf{s})$ interpolates between this lattice differential operator and the ``decoupled" differential ${\sf d^\ast d}(\mathbf{0})$. In what follows, we will not use this connection with lattice exterior calculus at all, so the reader unfamiliar with this terminology can safely treat~\eqref{eq:interaction_in_lattice_exterior_calculus_language} just as a formal definition.

        With~\eqref{eq:interaction_in_lattice_exterior_calculus_language}, we in fact have
        \begin{equation} \label{eq:writting_the_covariance_in_matrix_form}
            \beta \Cov_{\mathbf{s}}(\theta_x,\theta_y) = \big[m I + {\sf d}^\ast {\sf d}(\mathbf{s}) \big]^{-1} (x,y)
        \end{equation}
         for all $x,y\in \cE(\mathcal{G})$, where $I$ is the identity operator on $\ell^2(\cE(\mathcal{G}))$. Recall that for a matrix-valued smooth function $A(t)$ we have the identity
         \begin{equation} \label{eq:diffrentiating_the_inverse}
             \frac{\partial }{\partial t}A(t)^{-1} = -A(t)^{-1} \cdot \Big(\frac{\partial A(t)}{\partial t}\Big) \cdot  A(t)^{-1} \, .
         \end{equation}
         Hence, given any face ${\sf F} \in P^\prime(\mathcal{G}_1)$, we can differentiate using~\eqref{eq:diffrentiating_the_inverse} and see that
         \begin{equation} \label{eq:derivative_of_covariance_matrix_just_one}
             \frac{\partial}{ \partial s_{\sf F}} \big[m I + {\sf d}^\ast {\sf d}(\mathbf{s}) \big]^{-1} = -\big[m I + {\sf d}^\ast {\sf d}(\mathbf{s}) \big]^{-1} \cdot \Big(\frac{\partial {\sf d}^\ast {\sf d}(\mathbf{s})}{\partial s_{\sf F}}\Big) \cdot \big[m I + {\sf d}^\ast {\sf d}(\mathbf{s}) \big]^{-1} \, .
         \end{equation} 
         We shall denote by 
         \[
         A_{\sf F} = \frac{\partial}{\partial s_{\sf F}} {\sf d}^\ast {\sf d}(\mathbf{s}) \, ,
         \]
         while noting from~\eqref{eq:interaction_in_lattice_exterior_calculus_language} that $A_{\sf F}$ does not depend on $\mathbf{s}$ (see~\eqref{eq:def_of_sigma_p} for the definition of $\sigma_{p}(\mathbf{s})$). Therefore, repeated applications of~\eqref{eq:derivative_of_covariance_matrix_just_one} yields that for $\Ga = \{\gamma_1,\ldots,\gamma_n\} \subset P^\prime(\mathcal{G}_1)$ we have
         \begin{equation} \label{eq:derivative_of_covarinace_all_gamma}
             \partial^\Ga \big[m I + {\sf d}^\ast {\sf d}(\mathbf{s}) \big]^{-1} = (-1)^{n+1} \sum_{\sigma \in S_n}  \Big(\prod_{j=1}^n \big[m I + {\sf d}^\ast {\sf d}(\mathbf{s}) \big]^{-1} \cdot A_{\gamma_{\sigma(j)}} \Big) \cdot  \big[m I + {\sf d}^\ast {\sf d}(\mathbf{s}) \big]^{-1}  \, ,
         \end{equation}
         where $S_n$ is the set of all permutations of $\{1,\ldots,n\}$. Noting that $A_{\sf F}(q,q^\prime) = 0$ for edges $q,q^\prime \in \cE(\cG)$ which are not immediate neighbors on the face ${\sf F}$ or its neighboring faces, we observe from~\eqref{eq:derivative_of_covarinace_all_gamma} that 
         \begin{equation*}
             \big|\partial^\Ga \big[m I + {\sf d}^\ast {\sf d}(\mathbf{s}) \big]^{-1}(x,y)\big| \lesssim C^n \sum_{\substack{w:x\to y \\ w=\{q_1,\ldots,q_n\}}} \prod_{j=1}^{n} \Big| \big[m I + {\sf d}^\ast {\sf d}(\mathbf{s}) \big]^{-1}(q_j,q_{j+1}) \Big| 
         \end{equation*}
         where the sum over $w:x\to y$ consists of all paths $w= \{q_1,\ldots,q_n\}$ on length $n$ from $x$ to $y$ so that each face of $\Ga$ is touched exactly once. The reason for the (harmless) factor $C^n$ in the inequality above is that first we need to consider the  differences of correlations between $q_j$ and $q_{j+1}$, and the bound follows by expanding the product into a sum of size at most $C^n$. As the number of possible paths $w:x\to y$ is at most $(C L^{d-1})^{|\Ga|}$, we see from~\eqref{eq:writting_the_covariance_in_matrix_form} that
         \begin{align*}
            \beta \big|\Cov_{\mathbf{s}}(\theta_x,\theta_y)\big| & \lesssim \beta^{-1} (cL^{d})^{|\Ga|} \sup_{\substack{w:x\to y \\ w =\{q_1,\ldots,q_n\} } }  \, \prod_{j=1}^n\Big| \big[m I + {\sf d}^\ast {\sf d}(\mathbf{s}) \big]^{-1}(q_j,q_{j+1}) \Big| \\ & \stackrel{\text{Lemma}~\ref{lemma:proca_field_is_massive}}{\lesssim}  \beta^{-1} (cL^d)^{|\Ga|} \sup_{\substack{w:x\to y \\ w =\{q_1,\ldots,q_n\} } }  \, \prod_{j=1}^n C e^{-cm \, \text{dist}(q_j,q_{j+1})} \lesssim \beta^{-1} (CL^d)^{|\Ga|} \,  e^{-cm\, d(x,y;\Ga)} \, ,
         \end{align*}
         which is what we wanted to show. The bound for the expectation follows from the formula 
         \[
         \bE_{\mathbf{s}}^{\sf G}\big[\theta_x \big] = C_d \sum_{y\in \partial \mathcal{B}} \Cov_{\mathbf{s}}(\theta_x,\theta_y) \, \theta_y
         \]
         which is a simple consequence of the standard Gaussian conditioning formula. Here $C_d$ is the number of plaquettes which contain a specific edge, so $C_2 = 2$, $C_3 = 4$ and so on, though of course the exact number is not important to us. Altogether, since $|\theta_y| \le T_\beta \le 1$ for all $y\in \partial \mathcal{B}$, we are done. 
    \end{proof}
    Finally, we state and prove the following lemma, which shows how to apply $\partial^\Ga$ to the Gaussian measure $\mu_{\mathbf{s}}^{\sf G}$. It is in fact a simple consequence of the Gaussian integration by parts formula (sometimes referred to as Wick's theorem). For $\Ga\subset P^\prime(\mathcal{G}_1)$, we denote by $\mathcal{P}(\Ga)$ the set of all partitions of $\Ga$. 
    \begin{lemma} \label{lemma:gaussian_integration_by_parts}
        Let $F:\bR^{E(\La)} \to \bR$ be a smooth function which grows, along with all of its finite partial derivatives, at most polynomial at infinity. Then
        \[
        \partial^\Ga \bE_{\mathbf{s}}^{\sf G}\big[ F(\boldsymbol{\theta}) \big] = \sum_{\Pi\in \mathcal{P}(\Ga)} \bE_{\mathbf{s}}^{\sf G}\Big[ \big(\prod_{\pi\in \Pi} D^\pi\big) F(\boldsymbol{\theta}) \Big] \, ,
        \]
        where
        \begin{equation} \label{eq:def_of_D_pi}
         D^\pi F = \frac{1}{2} \sum_{x,y\in \cE(\cG)} \big(\partial^\pi \Cov_{\mathbf{s}} (\theta_x,\theta_y) \big) \, \frac{\partial^2 F}{ \partial \theta_x \partial \theta_y} + \sum_{x\in \cE(\cG)} \big(\partial^\pi \, \bE_{\mathbf{s}}^{\sf G}[\theta_x]\big) \, \frac{\partial F}{\partial \theta_x} \, .
        \end{equation}
    \end{lemma}
    \begin{proof}
        We first reduce the lemma to a simple statement about interpolating Gaussian measures. Indeed, let $\Sigma_0, \Sigma_1$ be two $N\times N$ positive-definite matrices and denote by
        \[
        \Sigma_t = \big(t\Sigma_1^{-1} + (1-t)\Sigma_0^{-1}\big)^{-1} \, , \qquad t\in [0,1]\, .
        \]
        Let $X_t$ be a Gaussian random vector in $\bR^N$ with mean-zero and covariance matrix $\Sigma_t$. Then, for any smooth text function $F$ as in the assumption of the lemma we have
        \begin{equation} \label{eq:interpolating_gaussians_formula}
            \frac{\partial}{ \partial t} \bE\big[F(X_t) \big] = \frac{1}{2} \sum_{i,j=1}^{N} \frac{\partial \Sigma_t}{\partial t} (i,j) \cdot \bE\Big[ \frac{\partial}{\partial x_i} \frac{\partial}{\partial x_j} F(X_t) \Big] \, .
        \end{equation}
        Assuming~\eqref{eq:interpolating_gaussians_formula}, the proof of the lemma follows by a repeated application of Liebnitz rule. We note that the expectation term appears in~\eqref{eq:def_of_D_pi} as we first need to shift the Gaussian measure $\mu_{\mathbf{s}}^{\sf G}$ to have zero mean, apply~\eqref{eq:interpolating_gaussians_formula} and then shift back again with an extra derivative. 

        To prove~\eqref{eq:interpolating_gaussians_formula}, we denote by $p_t(x)$ the density of $X_t$, and by 
        \begin{equation} \label{eq:proof_of_gaussian_IBP_def_of_A}
            A = \Sigma_1^{-1} - \Sigma_0^{-1} = \frac{\partial}{\partial t} \Sigma_t^{-1} \, .
        \end{equation}
        We have
        \[
        \log p_t(x) = C_N + \frac{1}{2} \log \det\big( \Sigma_t^{-1}\big) - \frac{1}{2} \langle x,\Sigma_t^{-1}x \rangle \, , 
        \]
        which implies, via the Jacobi identity, that
        \[
        \frac{\partial}{ \partial t} \log p_t(x) = \frac{1}{2} \Big( \text{Tr}(\Sigma_t \cdot A) - \langle x, A x \rangle   \Big) = \frac{1}{2} \sum_{\ell,k=1}^{N} A (k,\ell) \, \Big( \Sigma_t (k,\ell) - x_k x_\ell  \Big) \, .
        \]
        By our assumptions on $F$, we can apply the Dominated Convergence Theorem and exchange expectation and derivative freely. We get that 
        \begin{align} \label{eq:proof_of_lemma:gaussian_integration_by_parts_after_derivative_of_density} \nonumber 
             \frac{\partial}{ \partial t} \bE\big[F(X_t) \big] &= \int_{\bR^N} F(x) \, \frac{\partial}{ \partial t} p_t(x) \, {\rm d} x \\ \nonumber  &= \int_{\bR^N} F(x) \, p_t(x) \, \frac{\partial}{ \partial t} \log p_t(x) \, {\rm d} x \\ &= \frac{1}{2} \sum_{\ell,k=1}^{N} A(k,\ell ) \cdot \bE\Big[F(X_t) \Big( \Sigma_t (k,\ell) - (X_t)_k \,  (X_t)_\ell  \Big) \Big]
        \end{align}
        Recall that by Isserlis's theorem (sometimes also referred to as Wick's probability theorem, see~\cite[Chapter~1]{Janson-book}) we have
        \begin{equation*}
            \bE\big[F(X_t) (X_t)_k \big] = \sum_{j=1}^{N} \Sigma_t(k,j) \cdot \bE\Big[\frac{\partial}{\partial x_j} F(X_t)\Big] \, .
        \end{equation*}
        Applying this formula twice, we arrive at
        \begin{align*}
            \bE\Big[F(X_t) \, (X_t)_k \,  (X_t)_\ell  \Big] &= \sum_{j=1}^{N} \Sigma_t(k,j) \cdot \bE\Big[\frac{\partial}{\partial x_j} \Big( F(X_t) \,  (X_t)_\ell \Big) \Big] \\ &= \sum_{j=1}^{N} \Sigma_t(k,j) \cdot \Big( \sum_{i=1}^{N} \Sigma_t(i,\ell) \cdot  \bE\Big[\frac{\partial^2}{\partial x_j \partial x_i}  F(X_t)  \Big] + \delta_{j=\ell} \bE\big[F(X_t)\big]  \Big)  \\ &= \sum_{i,j=1}^{N} \Sigma_t(k,j) \, \Sigma_t(i,\ell) \, \bE\Big[\frac{\partial^2}{\partial x_j \partial x_i}  F(X_t)  \Big] + \Sigma_t(k,\ell)\,  \bE\big[F(X_t)\big] \, .
        \end{align*}
        Plugging the above into~\eqref{eq:proof_of_lemma:gaussian_integration_by_parts_after_derivative_of_density}, we get that 
        \begin{align*}
        \frac{\partial}{ \partial t} \bE\big[F(X_t) \big]  &= -  \frac{1}{2} \sum_{k,\ell=1}^{N} A(k,\ell) \cdot \Big( \sum_{i,j=1}^{N} \Sigma_t(k,j) \, \Sigma_t(i,\ell) \, \bE\Big[\frac{\partial^2}{\partial x_j \partial x_i}  F(X_t)  \Big] \Big)  \\ & = -\frac{1}{2} \sum_{i,j=1}^{N}   (\Sigma_t \cdot A \cdot \Sigma_t)(i,j) \cdot \bE\Big[\frac{\partial^2}{\partial x_j \partial x_i}  F(X_t)  \Big] \\  &\stackrel{\eqref{eq:proof_of_gaussian_IBP_def_of_A}}{=}  -\frac{1}{2} \sum_{i,j=1}^{N}   (\Sigma_t \cdot \frac{\partial \Sigma_t^{-1}}{\partial t} \cdot \Sigma_t)(i,j) \cdot \bE\Big[\frac{\partial^2}{\partial x_j \partial x_i}  F(X_t)  \Big] \\ &
         \stackrel{\eqref{eq:diffrentiating_the_inverse}}{=} \frac{1}{2} \sum_{i,j=1}^{N} \frac{\partial\Sigma_t}{\partial t}  (i,j) \cdot \bE\Big[\frac{\partial^2}{\partial x_j \partial x_i}  F(X_t)  \Big] \, .
        \end{align*}
    This proves~\eqref{eq:interpolating_gaussians_formula} and with that, the proof of the lemma is complete.
    \end{proof}
    \subsection{Proof of Lemma~\ref{lemma:GJS_weights_are_small}}
    By Claim~\ref{claim:condition_for_which_W_2_non_zero}, we only need to consider $\Ga\subset P^\prime(\mathcal{G}_1)$ which are connected and touch $\cB_1$ (otherwise, $W_2(\Ga;\cG) = 0$). We fix such $\Ga\subset P^\prime(\mathcal{G}_1)$ and aim to bound
    \[
    W_1(\Ga;\mathcal{G}) = \int_{[0,1]^{\Ga}} \partial^\Ga \log Z_{\mathcal{G},\eta}^{\sf G} (\mathbf{s}_\Ga) \, {\rm d} s_{\Ga}\, . 
    \]
    To do so, fix some face $\gamma \in \Ga$ and note that 
    \begin{equation*}
        \frac{\partial}{\partial s_\gamma} \log Z_{\mathcal{G},\eta}^{\sf G} (\mathbf{s}_\Ga) = -\bE_{\mathbf{s}}^{\sf G}\Big[ \frac{\partial}{\partial s_\gamma} S(\boldsymbol{\theta}; \mathbf{s}) \Big] \, ,
    \end{equation*}
    where we recall that $S(\cdot;\mathbf{s}) = S_{\La,\mathcal{G}}(\cdot;\mathbf{s})$ is the interpolated Gaussian action~\eqref{eq:def_of_gaussian_action_with_s}. In particular, since $S(\cdot;\mathbf{s})$ is linear in $s_\gamma$, the expectation above depends on $\mathbf{s}$ only through the underlying Gaussian measure $\mu_{\mathbf{s}}^{\sf G}$, and not through the integrand. Next, we want to preform the other derivatives from $\widetilde \Ga = \Ga \setminus \{\gamma\}$. For that, Lemma~\ref{lemma:gaussian_integration_by_parts} shows that 
    \begin{equation} \label{eq:proof_of_lemma_on_GJS_weights_after_derivative_tilde_gamma}
        \partial^{\widetilde \Ga} \bE_{\mathbf{s}}^{\sf G}\Big[ \frac{\partial}{\partial s_p} S(\boldsymbol{\theta}; \mathbf{s}) \Big] = \sum_{\Pi\in \mathcal{P}(\widetilde{\Ga})} \bE_{\mathbf{s}}^{\sf G} \Big[ \big(\prod_{\pi \in \Pi} D^\pi \big) \frac{\partial S(\boldsymbol{\theta}, \mathbf{s})}{ \partial s_\gamma} \Big] \, .
    \end{equation}
    Since $S(\boldsymbol{\theta},\mathbf{s})$ is a quadratic function in $\boldsymbol{\theta}$, the only partitions in~\eqref{eq:proof_of_lemma_on_GJS_weights_after_derivative_tilde_gamma} for which the summand is non-zero are those partitions $\Pi\in \mathcal{P}(\widetilde{\Ga})$ with at most two elements. There are at most $2^{|\Ga|}$ such partitions, and we get the bound 
    \[
    \Big|\partial^\Ga \log Z_{\mathcal{G},\eta}^{\sf G} (\mathbf{s}_\Ga) \Big| \le 2^{|\Ga|}  \sup_{\Pi\in \mathcal{P}(\widetilde{\Ga})} \bigg| \bE_{\mathbf{s}}^{\sf G} \Big[ \big(\prod_{\pi \in \Pi} D^\pi \big) \frac{\partial S(\boldsymbol{\theta}, \mathbf{s})}{ \partial s_\gamma} \Big]  \bigg| \, .
    \]
    To proceed, we note that the definition~\eqref{eq:def_of_D_pi} of $D^\pi$ along with Lemma~\ref{lemma:estimate_on_derivative_of_covariance} implies that 
    \begin{equation} \label{eq:proof_of_GJS_lemma_after_application_of_esitmates_on_derivatives}
    \Big|\partial^\Ga \log Z_{\mathcal{G},\eta}^{\sf G} (\mathbf{s}_\Ga) \Big| \lesssim  \beta^{-1} (CL^d)^{|\Ga|} e^{-cm \,  d(\Ga)} \, .    
    \end{equation}
    where $d(\Ga)$ is the shortest lattice path connecting the different faces in $\Ga$. Indeed, regardless of how we partition $\Ga$, the terms $\partial^\pi \Cov_{\mathbf{s}} (\theta_x,\theta_y)$ and $\partial^\pi \, \bE_{\mathbf{s}}^{\sf G}[\theta_x]$ in~\eqref{eq:def_of_D_pi} forces (via Lemma~\ref{lemma:estimate_on_derivative_of_covariance}) to consider paths which visit all different faces of $\pi$, and by multiplying with much consider paths visiting all faces of $\Ga$. All in all, since 
    \[
    d(\Ga) \ge \frac{L}{2d} |\Ga| \, ,
    \]
    we obtain from~\eqref{eq:proof_of_GJS_lemma_after_application_of_esitmates_on_derivatives} that
    \[
    |W_1(\Ga;\mathcal{G})| \lesssim (CL^{d})^{|\Ga|} \, e^{-cm L |\Ga|} 
    \]
    for some (dimension-dependent) constants $c,C>0$. Since
    \[
    C L^{d} e^{-cm L} 
    \]
    can be made as small as we want by taking $L\ge L_0(r,m,d)$ large enough, Claim~\ref{claim:bound_on_the_number_of_connected_sets} together with the bound above shows that 
    \[
    \| W_2(\cdot,\mathcal{G}) \|_r \lesssim \| W_1(\cdot,\mathcal{G}) \|_r  \le C_L \, ,
    \]
    as desired. 
    \qed 
    
    \section{Dealing with the non-Gaussian correction}
    \label{sec:dealing_with_non_gaussian_correction}
    
    It remains to prove Lemma~\ref{lemma:bound_for_the_non_gaussian_weight}, which we do in this section. Recall that for $\Delta \subset P^\prime(\mathcal{G}_1)$ we defined
    \begin{equation*}
        K_1(\Delta;\mathcal{G}) = \frac{1}{\Xi_{\mathcal{G},\eta}(\mathbf{0},+)} \int_{[0,1]^\Delta} \partial^\Delta \Xi_{\mathcal{G},\eta}(\mathbf{s}_{\Delta}) \, {\rm d} \mathbf{s}_\Delta \, ,
    \end{equation*}
    where $\Xi_{\mathcal{G},\eta}(\mathbf{0},+)$ is given by~\eqref{eq:def_of_Xi_0_+} and 
    \begin{equation*}
    \Xi_{\mathcal{G},\eta}(\mathbf{s}) = \int_{\bR^{\cE(\mathcal{G})}} e^{-\beta V_{\Lambda,\mathcal{G}}(\boldsymbol{\theta};\mathbf{s})} \chi_{\mathcal{G}}(\boldsymbol{\theta}) \, {\rm d}\mu_{\mathcal{G},\eta,\mathbf{s}}^{\sf G}  =  \bE_{\mathcal{G},\eta,\mathbf{s}}^{\sf G} \Big[ e^{-\beta V_{\Lambda,\mathcal{G}}(\boldsymbol{\theta};\mathbf{s})} \chi_{\mathcal{G}}(\boldsymbol{\theta})  \Big] \, .
    \end{equation*}
     Further, recall that for a face ${\sf F}\in P^\prime(\cG)$ we have 
     \begin{equation*}
         (\mathbf{s}_\Delta)_{\sf{F}} = \begin{cases}
             s_{\sf F} & \text{if} \ {\sf F} \in \Delta  \, , \\ 0 & \text{otherwise.}
         \end{cases}
     \end{equation*}
    
    \subsection{Reducing to a bound on derivatives of $\Xi$}
    
    Throughout this section, we will denote by 
    \begin{equation}
    \label{eq:def_of_mathcal_X}
        \mathcal{X} = X(\Delta)\cup (\mathcal{B}_1 \setminus \mathcal{B}) \, . 
    \end{equation}
    As we shall see below, $\mathcal{X}$ is the collection of $L$-blocks which contribute to $K_1(\Delta;\mathcal{G})$ non-trivially. This will be made more precise in a moment, but for now let us state a claim telling us how to handle the contribution from $\mathcal{X}$. 
    \begin{claim}\label{claim:estimate_of_derivative_non_gaussian_correction_X}
        Let $\Delta\subset P^\prime(\mathcal{G}_1)$ and let $\mathcal{X}$ be given by~\eqref{eq:def_of_mathcal_X}. Then for all $R \ge 1$ and for all $L$ large enough there exists $C_L>0$ so that  
        \[
        \bigg| \partial^\Delta \int_{\bR^{\cE(\mathcal{X})}} e^{-\beta V_{\Lambda,\mathcal{X}}(\boldsymbol{\theta};\mathbf{s}_\Delta)} \chi_{\mathcal{G}}(\boldsymbol{\theta}) \, {\rm d}\mu_{\mathcal{X},\eta,\mathbf{s}_\Delta}^{\sf G} \bigg| \le \nu(\beta) \, \exp\Big(-R|X(\Delta)| + C_L |\mathcal{B}_1|\Big)  \, , 
        \]
        where
        \[
        \lim_{\beta \to \infty} \nu(\beta) = 0
        \]
        uniformly in $R$ and $\Delta$. 
    \end{claim}
    Assuming Claim~\ref{claim:estimate_of_derivative_non_gaussian_correction_X} for the moment, we show how it yields the proof of Lemma~\ref{lemma:bound_for_the_non_gaussian_weight}.
    
    \begin{proof}[Proof of Lemma~\ref{lemma:bound_for_the_non_gaussian_weight}]
        We start the proof by making a simple observation about the weights $K_1$: 
        Although it appears from the definition (see~\eqref{eq:def_of_K_1}) that $K_1$ appears as a ratio of `large volume' quantities, it is really a `small volume' quantity, due to many cancellations between the numerator and the denominator. Indeed, since $(s_\Delta)_{\sf F} = 0$ for all faces ${\sf F}\not\in \Delta$, both numerator and denominator in~\eqref{eq:def_of_K_1} factor into contributions from $L$-blocks filling $\mathcal{B}_1\setminus \{L\text{-blocks connected by } \Delta \}$ which cancel out exactly. Therefore, the non-trivial contributions are coming either from $L$-blocks connected by $\Delta$, or from $L$-blocks that touch $\mathcal{B}_1$. 
       Put differently, $L$-blocks that do not touch $\mathcal{X}$ cancel out exactly. In turn, each remaining $L$-block in the denominator is lower bounded by some $c_L>0$, as seen by the inequality~\eqref{eq:two_sided_bound_for_corrector_in_given_L_block}. While this is fine to deal with the $L$-blocks from $\mathcal{B}_1$, to deal also with $L$-blocks from $X(\Delta)$ we need a slightly more accurate lower bound. Indeed, given any symbolic $L$-block $Q=Q(v)$ for $v\in X(\Delta)$, the lower bound in~\eqref{eq:proof_of_claim:two_sided_inequality_for_ratio_of_xi_role_of_chi} implies that 
        \begin{align*}
        \int_{\bR^{E(Q)}} \chi_{Q}(\boldsymbol{\theta})  e^{-\beta V_{Q}(\boldsymbol{\theta})} \, {\rm d}\mu_{Q,f}^{\sf G}(\boldsymbol{\theta}) &\ge \int_{\bR^{E(Q)}} \mathbf{1} \big\{\forall e\in Q \, , \  |\theta_e| \le T_\beta  \big\} \, {\rm d}\mu_{Q,f}^{\sf G}(\boldsymbol{\theta})   \\ &= 1- \int_{\bR^{E(Q)}} \mathbf{1} \big\{\exists e\in Q \, , \  |\theta_e| \ge T_\beta  \big\} \, {\rm d}\mu_{Q,f}^{\sf G}(\boldsymbol{\theta}) \, .
        \end{align*}
        By the union bound
        \[
        \int_{\bR^{E(Q)}} \mathbf{1} \big\{\exists e\in Q \, , \  |\theta_e| \ge T_\beta  \big\} \, {\rm d}\mu_{Q,f}^{\sf G}(\boldsymbol{\theta}) \lesssim |E(Q)| \int_{\beta T_\beta}^\infty e^{-x^2} \, {\rm d}x \le C_L \, e^{-\log^{2d}(\beta)}\le \frac{1}{4} \, ,
        \]
        for all $\beta$ large enough, and we conclude that 
        \[
        \int_{\bR^{E(Q)}} \chi_{Q}(\boldsymbol{\theta})  e^{-\beta V_{Q}(\boldsymbol{\theta})} \, {\rm d}\mu_{Q,f}^{\sf G}(\boldsymbol{\theta}) \ge \frac{1}{2} \, .
        \]
        In view of the cancellations discussed in the first paragraph of the proof, we arrive at the bound 
        \begin{multline*}
        |K_1(\Delta;\mathcal{G})| \\ \le 2^{|\mathcal{B}_1| + |X(\Delta)|} \,  \bigg| \partial^\Delta \int_{\bR^{\cE(\mathcal{X})}} e^{-\beta V_{\Lambda,\mathcal{X}}(\boldsymbol{\theta};\mathbf{s}_\Delta)} \chi_{\mathcal{G}}(\boldsymbol{\theta}) \, {\rm d}\mu_{\mathcal{X},\eta,\mathbf{s}_\Delta}^{\sf G} \bigg|   \stackrel{\text{Claim}~\ref{claim:estimate_of_derivative_non_gaussian_correction_X}}{\le}  \nu(\beta) \,  e^{-(R-2)|X(\Delta)| + C_L |\mathcal{B}_1|} \, .    
        \end{multline*}
        Hence, by taking $C_L^\prime = C_L + 1$ with $C_L$ gives as above and $R\ge R_0(d,r)$ large enough, we get that
        \begin{multline*}
             e^{-C_L^\prime|\mathcal{B}_1|}\sum_{\Delta\in \mathcal{C}(\mathcal{B}_1 ) } |K_1(\Delta;\mathcal{G})| \,  e^{r|X(\Delta)|} \\ \le \nu(\beta) \,  e^{-|\mathcal{B}_1|} \sum_{\Delta\in \mathcal{C}(\mathcal{B}_1)} e^{-(R-2 - r)|X(\Delta)|}   \stackrel{\text{Claim}~\ref{claim:bound_on_the_number_of_connected_sets}}{\lesssim } \nu(\beta) \, e^{-|\mathcal{B}_1|} \, |\mathcal{B}_1| \le \nu(\beta) \xrightarrow{\beta\to\infty} 0 \, ,
        \end{multline*}
        which is what we wanted to show. Finally, to deal with the case $\mathcal{B} = \mathcal{B}_1 = \emptyset$, we note that in this case Claim~\ref{claim:estimate_of_derivative_non_gaussian_correction_X} reads that
        \[
        \bigg| \partial^\Delta \int_{\bR^{\cE(\mathcal{X})}} e^{-\beta V_{\Lambda,\mathcal{X}}(\boldsymbol{\theta};\mathbf{s}_\Delta)} \chi_{\mathcal{G}}(\boldsymbol{\theta}) \, {\rm d}\mu_{\mathcal{X},f,\mathbf{s}_\Delta}^{\sf G} \bigg| \le \nu(\beta) \, e^{-R|X(\Delta)|} \, .
        \]
        Repeating the above argument, we get that
        \[
        |K_1(\Delta; \La^\prime)| \le \nu(\beta) \, 2^{|X(\Delta)|} e^{-R|X(\Delta)|} 
        \]
        and hence, for any $x\in L\bZ^d$ we have 
        \[
        \sum_{\substack{\Delta \ \text{connected} \\ \Delta \  \text{touch } x }} K_1(\Delta;\La^\prime) \,  e^{r|X(\Ga)|}  \le \nu(\beta) \sum_{\substack{\Delta \ \text{connected} \\ \Delta \  \text{touch } x }} e^{-(R-r-\log 2) |X(\Delta)|} \stackrel{\text{Claim}~\ref{claim:bound_on_the_number_of_connected_sets}}{\lesssim } \nu(\beta) \, .
        \]
        That is,
        \[
        \lim_{\beta\to \infty} \| K_1(\cdot; \La^\prime) \|_r = 0
        \]
        and we are done. 
    \end{proof}
    
    \subsection{Estimating the relevant derivatives}
    It remains to prove Claim~\ref{claim:estimate_of_derivative_non_gaussian_correction_X}. The technical challenge here is that the integral we are trying to bound depends on $\mathbf{s}_{\Delta}$ both through the function $V_{\La,\mathcal{X}}$ and through the underlying Gaussian measure. Therefore, when applying the derivatives $\partial^\Delta$ we will need to take this into account. Indeed, we can apply Leibnitz rule, and observe that
    \begin{multline} \label{eq:derivative_of_Xi_wrt_Delta_after_leibnitz}
    \partial^\Delta \int_{\bR^{\cE(\mathcal{X})}} e^{-\beta V_{\Lambda,\mathcal{X}}(\boldsymbol{\theta};\mathbf{s}_\Delta)} \chi_{\mathcal{G}}(\boldsymbol{\theta}) \, {\rm d}\mu_{\mathcal{X},\eta,\mathbf{s}_\Delta}^{\sf G} \\  = \sum_{\Upsilon\subset \Delta} \partial_{{\sf G}}^{\Upsilon}  \, \bE_{\mathcal{X},\eta,\mathbf{s}_{\Delta}}^{\sf G}\Big[ \Big( \prod_{{\sf F} \in \Delta\setminus \Upsilon} \sum_{\substack{p\in P(\La) \\ E(p)\cap {\sf F} \not= \emptyset }}\beta g({\sf d} \theta_p)\Big) \cdot e^{-\beta V_{\La,\mathcal{X}} (\boldsymbol{\theta}; \mathbf{s}_{\Delta})} \, \chi_{\mathcal{G}} (\boldsymbol{\theta}) \Big]
    \end{multline}
    where the symbol $\partial_{\sf G}^{\Upsilon}$ symbolizes taking the partial derivatives parametrized the faces in $\Upsilon\subset \Delta$ only with respect to the underlying Gaussian measure (and not with respect to $V_{\La,\mathcal{X}}$). 
    For pedagogical reasons, we deal with those derivatives in a separate claim. 
    \begin{claim} \label{claim:bound_of_derivative_of_non_gaussian_correction_with_repsect_to_gaussian_measure}
        For any $\Upsilon\subset \Delta$, $\Pi\in \mathcal{P}(\Upsilon)$ and for all $L$ large enough, we have
        \[
        \bigg| \,  \bE_{\mathcal{X},\eta,\mathbf{s}_{\Delta}}^{\sf G} \Big[ \big( \prod_{\pi\in \Pi} D^\pi\big) e^{-\beta V_{\La,\mathcal{X}} (\boldsymbol{\theta}; \mathbf{s}_{\Delta})} \, \chi_{\mathcal{G}} (\boldsymbol{\theta})  \Big] \bigg| \le C^{|\mathcal{X}|} \prod_{\pi\in \Pi} \big(\nu(\beta) \, (L^{d})^{|\pi|} \, e^{-c d(\pi)}\big)\, , 
        \]
        where $D^\pi$ is given by~\eqref{eq:def_of_D_pi}, $d(\pi)$ is the length of the shortest path in $\bZ^d$ which visits all faces of $\pi$ and 
        \[
        \lim_{\beta\to \infty} \nu(\beta) = 0 \, ,
        \]
        uniformly in $\Delta$. 
    \end{claim}
    
    \begin{proof}[Proof of Claim~\ref{claim:estimate_of_derivative_non_gaussian_correction_X}]
        Recall that the presence of the smooth cut-off $\chi_\mathcal{G}$ forces all $e\in \cE(\mathcal{G})$ to have $|\theta_e| \le 2 T_\beta$, where $T_\beta$ is given by~\eqref{eq:def_of_T_beta}. In particular, for plaquettes $p\in P^\prime(\mathcal{G}_1)$ we have 
        \[
        \beta |g({\sf d}\theta_p)| \lesssim \beta \,  (T_\beta)^4 \le \frac{\log^{4d+8}(\beta)}{\beta} \, .
        \]
        Plugging this bound into~\eqref{eq:derivative_of_Xi_wrt_Delta_after_leibnitz} and applying Lemma~\ref{lemma:gaussian_integration_by_parts} to deal with the $\partial_{\sf G}^{\Upsilon}$ derivatives, we arrive at 
        \begin{multline*}
            \bigg|\partial^\Delta \int_{\bR^{\cE(\mathcal{X})}} e^{-\beta V_{\Lambda,\mathcal{X}}(\boldsymbol{\theta};\mathbf{s}_\Delta)} \chi_{\mathcal{G}}(\boldsymbol{\theta}) \, {\rm d}\mu_{\mathcal{X},\eta,\mathbf{s}_\Delta}^{\sf G} \bigg| \\ \le C^{|\Delta|} \sup_{\Upsilon\subset \Delta} \bigg[ \Big(L^{d-1} \, \frac{\log^{4d+8}(\beta)}{\beta}\Big)^{|\Delta \setminus \Upsilon|} \sum_{\Pi\in \mathcal{P}(\Upsilon)} \bE_{\mathcal{X},\eta,\mathbf{s}_{\Delta}}^{\sf G} \Big[ \big( \prod_{\pi\in \Pi} D^\pi\big) e^{-\beta V_{\La,\mathcal{X}} (\boldsymbol{\theta}; \mathbf{s}_{\Delta})} \, \chi_{\mathcal{G}} (\boldsymbol{\theta})  \Big] \bigg] \, .
        \end{multline*}
        By Claim~\ref{claim:bound_of_derivative_of_non_gaussian_correction_with_repsect_to_gaussian_measure}, we can further bound the inner sum as
        \[
        \bigg|\sum_{\Pi\in \mathcal{P}(\Upsilon)} \bE_{\mathcal{X},\eta,\mathbf{s}_{\Delta}}^{\sf G} \Big[ \big( \prod_{\pi\in \Pi} D^\pi\big) e^{-\beta V_{\La,\mathcal{X}} (\boldsymbol{\theta}; \mathbf{s}_{\Delta})} \, \chi_{\mathcal{G}} (\boldsymbol{\theta})  \Big] \bigg| \le  C^{|\mathcal{X}|} \sum_{\Pi\in \mathcal{P}(\Upsilon)} \prod_{\pi\in \Pi} \Big(\nu(\beta) \, (L^d)^{|\pi|} \, e^{-c d(\pi)}\Big) \, .
        \]
        Since $|\mathcal{X}| \le |X(\Delta)| + |\mathcal{B}_1|$, we can combine both bounds and arrive at the inequality 
        \begin{multline} \label{eq:proof_of_claim:estimate_of_derivative_non_gaussian_correction_X_after_taking_derivatives}
             \bigg|\partial^\Delta \int_{\bR^{\cE(\mathcal{X})}} e^{-\beta V_{\Lambda,\mathcal{X}}(\boldsymbol{\theta};\mathbf{s}_\Delta)} \chi_{\mathcal{G}}(\boldsymbol{\theta}) \, {\rm d}\mu_{\mathcal{X},\eta,\mathbf{s}_\Delta}^{\sf G} \bigg| \\ \le C^{|X(\Delta)| + |\mathcal{B}_1|}  \sup_{\Upsilon\subset \Delta} \bigg[ \Big(L^{d-1} \, \frac{\log^{4d+8}(\beta)}{\beta}\Big)^{|\Delta \setminus \Upsilon|} \sum_{\Pi\in \mathcal{P}(\Upsilon)} \prod_{\pi\in \Pi} \Big(\nu(\beta) \, (L^d)^{|\pi|} \,  e^{-c d(\pi)}\Big) \bigg] \, .
         \end{multline}
         Now, for any $R \ge 1$ we have 
         \begin{multline*}
             \sum_{\Pi\in \mathcal{P}(\Upsilon)} \prod_{\pi\in \Pi} \Big(\nu(\beta) \, (L^d)^{|\pi|} \,  e^{-c d(\pi)}\Big) \\ = e^{-R|\Upsilon|} \sum_{\Pi\in \mathcal{P}(\Upsilon)} \prod_{\pi\in \Pi} \Big(\nu(\beta) \, e^{R|\pi|} \, (L^d)^{|\pi|} \,  e^{-c d(\pi)}\Big)  \\ \stackrel{\text{Claim}~\ref{claim:bound_on_function_on_collection_of_subsets}}{\le} e^{-R|\Upsilon|} \,  \nu(\beta) \,  \exp\Big(|\Upsilon| \cdot  \| (L^d)^{|\cdot|} \, e^{-cd(\cdot)} \|_{R}\Big) \, .
         \end{multline*}
         Since $d(\pi) \ge c L |\pi|$, we can choose $L$ large enough so that $\|(L^d)^{|\cdot|} \,  e^{-cd(\cdot)} \|_{R} \le 1$, and get that 
         \[
           \sum_{\Pi\in \mathcal{P}(\Upsilon)} \prod_{\pi\in \Pi} \Big(\nu(\beta) \, (L^d)^{|\pi|} \,  e^{-c d(\pi)}\Big)  \le e^{-(R-1)|\Upsilon|} \,  \nu(\beta) \, . 
         \]
         Plugging the above into~\eqref{eq:proof_of_claim:estimate_of_derivative_non_gaussian_correction_X_after_taking_derivatives} and taking $\beta$ so large such that 
         \[
         L^{d-1} \, \frac{\log^{4d+8}(\beta)}{\beta} \le e^{-2R} \, ,
         \]
         we obtain
         \[
         \bigg|\partial^\Delta \int_{\bR^{\cE(\mathcal{X})}} e^{-\beta V_{\Lambda,\mathcal{X}}(\boldsymbol{\theta};\mathbf{s}_\Delta)} \chi_{\mathcal{G}}(\boldsymbol{\theta}) \, {\rm d}\mu_{\mathcal{X},\eta,\mathbf{s}_\Delta}^{\sf G} \bigg|  \le \nu(\beta) \, C^{|X(\Delta)| + |\mathcal{B}_1|}   \, e^{-R|\Delta|} 
         \]
         from which the claim follows. 
    \end{proof}
    \begin{proof}[Proof of Claim~\ref{claim:bound_of_derivative_of_non_gaussian_correction_with_repsect_to_gaussian_measure}]
        The first step in the proof is to expand the differential operator $(\prod_{\pi\in \Pi} D^\pi)$, and to do so we introduce some notion. As the partition $\Pi$ remains fixed for the proof we write it as $\Pi = \{\pi_1,\ldots,\pi_N\}$. Furthermore, for edges $x,y,z\in \cE(\mathcal{X})$ we write 
        \[
        {\sf C}(x,y;\pi) = \frac{1}{2} \partial^\pi \Cov_{\mathbf{s}} (\theta_x,\theta_y)  \, , \qquad  {\sf m}(z;\pi) = \partial^\pi \, \bE_{\mathbf{s}}^{\sf G}[\theta_z] \, .
        \]
        Let $\mathbf{x} = (x_1,\ldots,x_N) $ be a collection of edges (that is, with repetitions) from $\cE(\mathcal{X})$, and define $\mathbf{y}, \mathbf{z}$ similarly. 
        By utilizing~\eqref{eq:def_of_D_pi} and expanding the product, we see that 
        \begin{align} \label{eq:expanding_D_pi_general_partition}
            (\prod_{\pi\in \Pi} D^\pi) &\big(e^{-\beta V_{\La,\mathcal{X}}} \, \chi_{\mathcal{G}} \big)  = \sum_{\mathbf{x},\mathbf{y},\mathbf{z} \in \cE(\mathcal{X})^N} \sum_{\boldsymbol{\eps} \in \{0,1\}^N} \\ & \nonumber \Bigg(\bigg[\prod_{\ell=1}^{N} \Big( ({\sf C}(x_\ell,y_\ell;\pi_\ell) )^{\eps_\ell} ({\sf m}(z_\ell;\pi_\ell))^{1-\eps_\ell} \Big) \bigg]\times\bigg[\prod_{\ell=1}^{N} \Big(\frac{\partial^2}{\partial \theta_{x_\ell}\partial \theta_{y_\ell}} \Big)^{\eps_\ell} \cdot \Big(\frac{\partial}{\partial \theta_{ z_{\ell}}}\Big)^{1-\eps_\ell}\bigg] \big(e^{-\beta V_{\La,\mathcal{X}}} \, \chi_{\mathcal{G}} \big) \Bigg) \, .
        \end{align}
        By Lemma~\ref{lemma:estimate_on_derivative_of_covariance}, we can estimate the terms in the first product and obtain
        \begin{align} \label{eq:expanding_D_pi_general_partition_after_estimate_on_covariance}
            \bigg| \,  \bE_{\mathcal{X},\eta,\mathbf{s}_{\Delta}}^{\sf G} &\Big[ \big( \prod_{\pi\in \Pi} D^\pi\big)  e^{-\beta V_{\La,\mathcal{X}} (\boldsymbol{\theta}; \mathbf{s}_{\Delta})}  \, \chi_{\mathcal{G}} (\boldsymbol{\theta})  \Big] \bigg| \le \\ \nonumber  & C^{|\mathcal{X}|} \,  \sum_{\mathbf{x},\mathbf{y},\mathbf{z} \in \cE(\mathcal{X})^N} \sum_{\boldsymbol{\eps} \in \{0,1\}^N} \beta^{-N} \cdot \Big(\prod_{\ell=1}^{N} (c_2 L^d)^{|\pi_\ell|}\Big) \,  e^{-c_3 m \sum_{\ell=1}^{N} \big[\eps_\ell d(x_\ell,y_\ell;\pi_\ell) + (1-\eps_\ell) d(z_\ell, \mathcal{B}_1 ; \pi_\ell)\big] } \\ \nonumber & \qquad \qquad \qquad \qquad \qquad \qquad \times \bE_{\mathcal{X},\eta,\mathbf{s}_{\Delta}}^{\sf G} \big|\mathcal{D}(\mathbf{x},\mathbf{y},\mathbf{z}, \boldsymbol{\eps} ) \big(e^{-\beta V_{\La,\mathcal{X}}} \, \chi_{\mathcal{G}} \big)\big|  \, .
        \end{align}
        where, for a fixed $\mathbf{x},\mathbf{y},\mathbf{z}, \boldsymbol{\eps}$ occurring in the sum~\eqref{eq:expanding_D_pi_general_partition}, we denote by 
        \begin{equation} \label{eq:def_of_mathcal_D_differntial}
           \mathcal{D}(\mathbf{x},\mathbf{y},\mathbf{z}, \boldsymbol{\eps} ) = \prod_{\ell=1}^{N} \Big(\frac{\partial^2}{\partial \theta_{x_\ell}\partial \theta_{y_\ell}} \Big)^{\eps_\ell} \cdot \Big(\frac{\partial}{\partial \theta_{ z_{\ell}}}\Big)^{1-\eps_\ell} \, . 
        \end{equation}
        Furthermore, for a pair of edges $(x,y)$ or a single edge $z$, we will denote by
        \begin{equation*}
            {\sf n}(x,y) \, , \quad {\sf n}(z)\, , 
        \end{equation*}
        the number of times the pair $(x,y)$, or the edge $z$, appear in the differential operator~\eqref{eq:def_of_mathcal_D_differntial}, respectively. 
        \begin{claim} \label{claim:bound_of_expectation_of_derivative_mathcal_D}
            With the notation~\eqref{eq:def_of_mathcal_D_differntial}, we have
            \[
            \bE_{\mathcal{X},\eta,\mathbf{s}_{\Delta}}^{\sf G} \big|\mathcal{D}(\mathbf{x},\mathbf{y},\mathbf{z}, \boldsymbol{\eps} ) \big(e^{-\beta V_{\La,\mathcal{X}}} \, \chi_{\mathcal{G}} \big)\big| \le  C^{|\mathcal{X}|} \Big(\frac{\beta}{\log \beta}\Big)^N  \prod_{(x,y)\in \cE(\mathcal{X})^2} \big({\sf n}(x,y)!\big)^2 \prod_{z\in \cE(\mathcal{X})} \big( {\sf n}(z)!\big)^2 \, ,
            \]
            with the usual agreement that $0! = 1$. 
        \end{claim}
        \begin{claim} \label{claim:combinatorial_estimate_on_sum_of_factorials}
        For all $\eta>0$ we have 
          \[
          \prod_{(x,y)\in \cE(\mathcal{X})^2} \big({\sf n}(x,y)!\big)^2 \prod_{z\in \cE(\mathcal{X})} \big({\sf n}(z)! \big)^2 \le C_\eta^{|\mathcal{X}|} \exp\Big(\eta \sum_{\ell=1}^{N} \big[\eps_\ell d(x_\ell,y_\ell;\pi_\ell) + (1-\eps_\ell) d(z_\ell, \mathcal{B}_1 ; \pi_\ell)\big]  \Big) \, .
          \]
        \end{claim}
        \noindent
        By combining the bound~\eqref{eq:expanding_D_pi_general_partition_after_estimate_on_covariance} with Claim~\ref{claim:bound_of_expectation_of_derivative_mathcal_D} and Claim~\ref{claim:combinatorial_estimate_on_sum_of_factorials} (applied with $\eta = c_3m/2$), we arrive at 
        \begin{align} \label{eq:bound_on_expectation_of_D_pi_derivatives_after_claims} \nonumber
            \bigg| \, & \bE_{\mathcal{X},\eta,\mathbf{s}_{\Delta}}^{\sf G} \Big[ \big( \prod_{\pi\in \Pi} D^\pi\big)  e^{-\beta V_{\La,\mathcal{X}} (\boldsymbol{\theta}; \mathbf{s}_{\Delta})}  \, \chi_{\mathcal{G}} (\boldsymbol{\theta})   \Big] \bigg| \\ \nonumber & \le C^{|\mathcal{X}|} \,  \sum_{\mathbf{x},\mathbf{y},\mathbf{z} \in \cE(\mathcal{X})^N} \sum_{\boldsymbol{\eps} \in \{0,1\}^N} (\log(\beta))^{-N} \cdot \Big(\prod_{\ell=1}^{N} (c_2 L^d)^{|\pi_\ell|}\Big) \, e^{- c m \sum_{\ell=1}^{N} \big[\eps_\ell d(x_\ell,y_\ell;\pi_\ell) + (1-\eps_\ell) d(z_\ell, \mathcal{B}_1 ; \pi_\ell)\big] }  \\ &\le  C^{|\mathcal{X}|} \frac{1}{\big(\log (\beta) \big)^{|\Pi|}} \prod_{\pi\in \Pi} \Big(  \sum_{x,y\in \cE(\mathcal{X})} (c_2L^d)^{|\pi|} \, e^{- c m d(x,y;\pi)} + \sum_{z\in \cE(\mathcal{X})} (c_2L^d)^{|\pi|} \, e^{- c m d(z, \mathcal{B}_1 ; \pi)\big] } \, \Big) \, .
        \end{align}
        Using the exponential decay, we can estimate the sums in the product as 
        \[
        \max\bigg\{\sum_{x,y\in \cE(\mathcal{X})} (c_2L^d)^{|\pi|} \, e^{- c m d(x,y;\pi)} , \sum_{z\in \cE(\mathcal{X})} (c_2L^d)^{|\pi|} \, e^{- c m d(z, \mathcal{B}_1 ; \pi)\big] } \bigg\} \lesssim  (L^d)^{|\pi|} e^{-cd(\pi)} \, . 
        \]
        Plugging this estimate into~\eqref{eq:bound_on_expectation_of_D_pi_derivatives_after_claims}, we finally arrive at
        \[
        \bigg| \,  \bE_{\mathcal{X},\eta,\mathbf{s}_{\Delta}}^{\sf G} \Big[ \big( \prod_{\pi\in \Pi} D^\pi\big) e^{-\beta V_{\La,\mathcal{X}} (\boldsymbol{\theta}; \mathbf{s}_{\Delta})}  \, \chi_{\mathcal{G}} (\boldsymbol{\theta})   \Big] \bigg| \le C^{|\mathcal{X}|} \prod_{\pi \in \Pi} \Big( \frac{1}{\log(\beta)} (L^d)^{|\pi|} e^{-cd(\pi)}\Big) \, ,
        \]
        which yields the claim. 
    \end{proof}
    \begin{proof}[Proof of Claim~\ref{claim:bound_of_expectation_of_derivative_mathcal_D}]
        In this proof, we will collectively treat all partial derivatives appearing in $\mathcal{D}(\mathbf{x}, \mathbf{y}, \mathbf{z} , \boldsymbol{\eps})$ as indexed by $\mathbf{w} = (w_1,\ldots,w_{\widetilde N})$, as it will not be important to distinguish between the different roles on $(\mathbf{x},\mathbf{y})$ and $\mathbf{z}$. In turn, $\widetilde{N}$ will be the number of derivatives we have to take, and hence
        \[
        N \le \widetilde N\le 2N \, .
        \]
        By Leibnitz rule, we can write $\mathbf{w} = \{ \mathbf{w}^1,\mathbf{w}^2\}$ as a symbolic partition and note that
        \begin{equation} \label{eq:proof_of_claim:bound_of_expectation_of_derivative_mathcal_D_after_leibnitz}
            \bE_{\mathcal{X},\eta,\mathbf{s}_{\Delta}}^{\sf G} \big|\mathcal{D}(\mathbf{x},\mathbf{y},\mathbf{z}, \boldsymbol{\eps} ) \big(e^{-\beta V_{\La,\mathcal{X}}} \, \chi_{\mathcal{G}} \big)\big| \le 2^{\widetilde{N}} \max_{\mathbf{w} = \{ \mathbf{w}^1,\mathbf{w}^2\}} \sup_{\boldsymbol{\theta}} \Big|\mathcal{D} (\mathbf{w}^1) \big(e^{-\beta V_{\La,\mathcal{X}}} \big) \Big| \cdot \sup_{\boldsymbol{\theta}} \Big|\mathcal{D} (\mathbf{w}^2) \big(\chi_{\mathcal{G}} \big) \Big| \, , 
        \end{equation}
        where the maximum is over all possible partitions the index set of $\mathbf{w}$ into two disjoint sets. Let $\mathbf{w} = \{ \mathbf{w}^1,\mathbf{w}^2\}$ be any such partition, and denote by $N_1$ and $N_2$ the number of derivatives attached to $\mathcal{D}(\mathbf{w}^1)$ and $\mathcal{D}(\mathbf{w}^2)$, respectively. Since the function $\mathbf{\theta} \mapsto e^{-\beta V_{\La,\mathcal{X}}}$ is holomorphic, the Cauchy estimate~\cite[Theorem~2.2.7]{Hormander-SeveralComplexVariables} in the polydisk of radius $\log(\beta)/\sqrt{\beta}$ yields
        \begin{equation*}
            \sup_{\boldsymbol{\theta}} \Big|\mathcal{D} (\mathbf{w}^1) \big(e^{-\beta V_{\La,\mathcal{X}}} \big) \Big| \le \Big(\frac{\sqrt{\beta}}{\log \beta} \Big)^{N_1} \cdot \Big( \prod_{w\in \mathbf{w}^1} {\sf n}(w)! \Big) \cdot  \sup_{\boldsymbol{\theta}: |\theta_e| \le \frac{\log \beta}{\sqrt{\beta}}}e^{-\beta V_{\La,\mathcal{X}} (\boldsymbol{\theta})}  \, .
        \end{equation*}
        Since for all $\beta$ large enough we have
        \[
        e^{ - \beta g(\theta)  } \le e^{ C \beta \frac{\log^4(\beta)}{\beta^2}} \le 2
        \]
        in the relevant range, we get that 
        \begin{equation} \label{eq:bound_on_derivatives_from_w_1}
            \sup_{\boldsymbol{\theta}} \Big|\mathcal{D} (\mathbf{w}^1) \big(e^{-\beta V_{\La,\mathcal{X}}} \big) \Big| \le C^{|\mathcal{X}|} \,  \Big(\frac{\sqrt{\beta}}{\log \beta} \Big)^{N_1} \cdot \Big( \prod_{w\in \mathbf{w}^1} {\sf n}(w)! \Big) \, . 
        \end{equation}
        We remark that the bound~\eqref{eq:bound_on_derivatives_from_w_1} can be improved asymptotically in $\beta$, as the first derivative of $e^{-\beta V_{\La,\mathcal{X}}}$ (in all directions) also decays in $\beta$, but this will not be important in what follows so we use the simple bound obtained by Cauchy's estimates instead. By recalling the definition of $\chi_\mathcal{G}$ (see, in particular, item (3) in properties of the smooth bump function $\chi$ from Section~\ref{subsection:partition_of_unity}) we get that  
        \[
        \sup_{\boldsymbol{\theta}} \Big|\mathcal{D} (\mathbf{w}^2) \big(\chi_{\mathcal{G}} \big) \Big| \le C^{|\mathcal{X}|} \,  T_\beta^{-N_2} \, \prod_{w\in \mathbf{w}^2} ({\sf n}(w)!)^2 =  C^{|\mathcal{X}|} \,  \Big(\frac{\sqrt{\beta}}{\log^{2d+2}(\beta)}\Big)^{N_2} \, \prod_{w\in \mathbf{w}^2} ({\sf n}(w)!)^2  \, . 
        \]
        Plugging~\eqref{eq:bound_on_derivatives_from_w_1} and the above bound into~\eqref{eq:proof_of_claim:bound_of_expectation_of_derivative_mathcal_D_after_leibnitz} yields 
        \[
        \bE_{\mathcal{X},\eta,\mathbf{s}_{\Delta}}^{\sf G} \big|\mathcal{D}(\mathbf{x},\mathbf{y},\mathbf{z}, \boldsymbol{\eps} ) \big(e^{-\beta V_{\La,\mathcal{X}}} \, \chi_{\mathcal{G}} \big)\big| \le C^{|\mathcal{X}|} \Big(\frac{\beta}{\log^2(\beta)}\Big)^{N_1/2 + N_2/2} \prod_{(x,y)\in \mathcal{X}^2} \big({\sf n}(x,y)!\big)^2 \prod_{z\in \mathcal{X}} \big( {\sf n}(z)!\big)^2 \, ,
        \]
        and since $N_1 + N_2 = \widetilde N \le 2N$, the claim follows. 
    \end{proof}
    
    \begin{proof}[Proof of Claim~\ref{claim:combinatorial_estimate_on_sum_of_factorials}]
        The claim follows from the analogous claim in Glimm, Jaffe and Spencer's work (see~\cite[Lemma~10.2]{Glimm-Jaffe-Spencer}), but for the sake of readability we produce the proof here as well. We will only prove the relevant bound on the product of pairs $(x,y)\in \cE(\mathcal{X})^2$, as the same proof works also for the corresponding bound for the product over $z\in \cE(\mathcal{X})$. Without loss of generally, we assume that $\eps_\ell = 1$ for all $\ell=1,\ldots, N$ (otherwise, simply restrict to the indices $\eps_\ell = 1$, where the sam bound applies). 
        
        For any $\pi\in \Pi$ we denote by ${\sf p}(\pi)$ the pair $(x,y)\in \cE(\mathcal{X})^2$ which is matched to $\pi$ in the sum~\eqref{eq:expanding_D_pi_general_partition} (namely, the pair which appears in ${\sf C}(p(\pi);\pi)$ in the first product). By definition, we have 
        \begin{equation} \label{eq:def_of_n_x_y_in_terms_of_p}
        {\sf n}(x,y) = \# \big\{ \pi \in \Pi \, : \, {\sf p}(\pi) = (x,y) \big\} \, .    
        \end{equation}
        Furthermore, for any fixed pair $(x,y)\in \cE(\mathcal{X})^2$, we have the trivial bound
        \[
        \# \big\{ \pi \in \Pi \, : \, {\sf p}(\pi) = (x,y)  \quad \text{and} \quad d(x,y;\pi) \le r \big\} \lesssim r^d \, ,
        \]
        since the partition elements do not overlap. 
        Plugging $r= {\sf n}(x,y)^{1/d}$ yields that 
        \[
        {\sf n}(x,y)^{1/d} \le C \, d(x,y;\pi)
        \]
        for at least ${\sf n}(x,y)/2$ of the partition elements in the set~\eqref{eq:def_of_n_x_y_in_terms_of_p}. We thus have
        \begin{equation} \label{eq:bound_of_n_x_y_in_terms_of_d}
        {\sf n}(x,y)^{1+1/d} \le C \sum_{\pi \, : \,  {\sf p}(\pi) = (x,y)} d(x,y;\pi) \, ,    
        \end{equation}
        
        and, by the Stirling bound, we get
        \begin{align*}
            \prod_{(x,y)\in \cE(\mathcal{X})^2} \big({\sf n}(x,y)!\big)^2  & \le \exp\Big( C \sum_{(x,y) \in \cE(\mathcal{X})^2} {\sf n}(x,y) \log {\sf n}(x,y) \Big) \\ &\le C_\eta^{|\mathcal{X}|} \exp\Big( \eta \sum_{(x,y) \in \cE(\mathcal{X})^2} \big({\sf n}(x,y)\big)^{1+1/d} \, \mathbf{1}_{\{{\sf n}(x,y) \ge 1\}} \Big) \\ &\stackrel{\eqref{eq:bound_of_n_x_y_in_terms_of_d}}{\le}  C_\eta^{|\mathcal{X}|} \exp\Big( \eta \sum_{\ell=1}^{N}  d(x_\ell,y_\ell;\pi_\ell) \Big) \, .
        \end{align*}
        We remark that the reason we are getting the term $C_\eta^{|\mathcal{X}|}$ is the second inequality (as opposed to $C_\eta^{|\mathcal{X}|^2}$) is that the number of non-zero summands in the sum over $\cE(\mathcal{X})^2$ is at most $|\Pi| \le |\Delta| \le |\mathcal{X}|$. With that, the proof of the claim is complete and we are done. 
    \end{proof}

    \subsubsection*{Acknowledgment}
    This work was partially supported by NSF grants DMS-2413864 and DMS-2401136.
    
    \medskip
    \medskip

    \bibliographystyle{abbrv}
    \bibliography{YangMills}

\end{document}